\documentclass[11pt,reqno,a4paper,oneside]{amsart}

\usepackage{tikz}
\usepackage{verbatim}

\usepackage{latexsym}
\usepackage{amsmath}
\usepackage{amsthm}
\usepackage{amssymb}
\usepackage{amsfonts}
\usepackage{fancyhdr}
\usepackage{comment}

\usepackage{mathrsfs}                           

\newcommand{\mR}{\mathbb{R}}                    
\newcommand{\abs}[1]{\lvert #1 \rvert}          

\def\tilde{\widetilde}

\def\ID{{\mathbb{D}}}

\def\S2{{\bf S}(2)}

\def\1ton{1,2,\ldots,n}
\def\det{{\rm det}\,}

\def\II{{\bf I}}

\def\S{{\mathcal S}}

\def\p{\mbox{{\boldmath $p$}}}

\def \beq {\begin {eqnarray}}
\def \eeq {\end {eqnarray}}
\def \ba {\begin {eqnarray*}}
\def \ea {\end {eqnarray*}}

\def\p{\partial}

\newcommand{\closure}[1]{\overline{#1}}

\newtheorem{thm}{Theorem}[section]
\newtheorem{prop}[thm]{Proposition}

\newtheorem{lemma}[thm]{Lemma}
\theoremstyle{definition}
\newtheorem*{definition}{Definition}

\newtheorem*{remark}{Remark}

\numberwithin{equation}{section}

\title[Inverse problems for FEM and resistor networks]
{Inverse problems and invisibility cloaking for FEM models and resistor networks}

\author{Matti Lassas}
\address{Matti Lassas, Department of Mathematics and Statistics, Univ. of Helsinki,
Finland}

\author{Mikko Salo}
\address{Mikko Salo, Department of Mathematics and Statistics, Univ. of Jyv\"askyl\"a,
Finland}

\author{Leo Tzou}
\address{Leo Tzou, Department of Mathematics and Statistics, Stockholm Univ.,
Sweden}

\date{July 5, 2013}

\begin{document}

\begin{abstract} In this paper we consider inverse problems for resistor networks and for 
models obtained via the Finite Element Method (FEM) for the conductivity equation. 
These correspond to discrete versions of 
the inverse conductivity problem of Calder\'on.
We characterize
FEM models corresponding to a given triangulation of the domain 
that are equivalent to certain resistor networks, 
and apply the results to study nonuniqueness of the discrete inverse problem.
It turns out that the degree of nonuniqueness for the discrete problem is larger
than the one for the partial differential equation. We also study invisibility cloaking
for FEM models, and show how an arbitrary body can be surrounded with a layer so
that the cloaked body has the same boundary measurements as a given background medium.
\end{abstract}
\maketitle


\section{Introduction} \label{sec_intro}

\subsection{FEM and resistor networks}
In this paper we consider inverse problems for FEM models of the (anisotropic) conductivity equation, and for resistor networks. Let $\Omega \subset \mR^2$ be the interior of a simple polygon, let $\sigma \in L^{\infty}(\Omega, \mR^{2 \times 2})$ be a symmetric positive definite matrix function, and consider the Dirichlet problem for the conductivity equation 
$$
\nabla \cdot \sigma \nabla u = 0 \text{ in } \Omega, \qquad u|_{\partial \Omega} = f.
$$
If one is given a triangulation of $\closure{\Omega}$, then a Finite Element Method (FEM) model, based on piecewise affine basis functions, for this Dirichlet problem is a matrix equation 
\begin{equation} \label{fem_equation_intro}
\left( \begin{array}{cc} I_B & 0 \\ L' & L'' \end{array} \right) u = \left( \begin{array}{cc} f \\ 0 \end{array} \right)
\end{equation}
where $u = (u_1 \cdots u_N)^t \in \mR^N $ corresponds to values of the FEM solution at the vertices of the triangulation, $f = (f_1 \cdots f_B)^t$ corresponds to the boundary value $f$ at the boundary vertices, and the matrix $(L' \ L'')$ depends on $\sigma$ and the triangulation and has rows whose elements add up to zero.

On the other hand, if one considers the triangulation of $\closure{\Omega}$ as a graph, and if to each edge of the graph one assigns a resistor with a given conductivity, then the Dirichlet problem for this resistor network is to find a solution $u = (u_1 \cdots u_N)^t$, corresponding to voltages at each vertex, of the matrix equation 
\begin{equation}  \label{rn_equation_intro}
\left( \begin{array}{cc} I_B & 0 \\ \tilde{L}' & \tilde{L}'' \end{array} \right) u = \left( \begin{array}{cc} f \\ 0 \end{array} \right)
\end{equation}
where $f = (f_1 \cdots f_B)^t$ corresponds to voltages at the boundary vertices. 
Also here, the rows of the matrix $(\tilde{L}' \ \tilde{L}'')$, which depends on the resistors, add up to zero.

Thus, formally, the FEM model and resistor networks lead to the same kind of matrix equation.
This has also been pointed out in \cite{Lionheart}. In this paper we will show that there is a sort of equivalence between the two. Given an anisotropic conductivity in a triangulated domain $\Omega$, one can assign resistors (with possibly negative conductivity) to each edge in the triangulation, in such a way that the equations for the FEM model and resistor network are the same (and therefore have the same solution). Conversely, given a resistor network corresponding to a triangulated domain $\Omega \subset \mR^2$, one can find an anisotropic conductivity (not necessarily positive definite) which is constant in each triangle, such that the equations for the resistor network and the FEM model again coincide. Moreover we characterize all such conductivities for the FEM model that correspond to the given resistor network.

In this paper, we actually do not consider the inverse problem of recovering a conductivity from boundary measurements in the discrete models. Rather, we will use the equivalence of FEM models and resistor networks to describe a new type of nonuniqueness in the inverse problem for the FEM model. We will also describe invisibility cloaking constructions in the discrete case. These results highlight a certain difference when compared with known results for the continuum inverse problem, i.e.\ the inverse problem for  the conductivity equation $\nabla\,\cdotp\sigma \nabla u = 0$, which is a partial differential equation (PDE). To discuss the difference, let us first recall some results related to the PDE problem.

\subsection{Inverse problem for the conductivity equation.}
The inverse conductivity problem, also known as Calder\'on's inverse problem, is the question  whether an  unknown conductivity distribution inside a domain
can be determined  from voltage and current
measurements made on its boundary.
The measurements correspond to the knowledge of the
Dirichlet-to-Neumann map  $ \Lambda_\sigma $ (the continuous voltage-to-current map) associated with a matrix conductivity $\sigma = (\sigma^{jk})_{j,k=1}^2$. This map takes the Dirichlet boundary
values of the solution of the conductivity equation
\begin{equation}\label{johty}
\nabla\cdot \sigma(x) \nabla u(x) =\sum_{j,k=1}^2 
\frac \p{\p x^j}\bigg(\sigma^{jk}(x) \frac \p{\p x^k} u(x) \bigg)= 0
\end{equation}
 to the corresponding Neumann boundary values,
\beq\label{eq: DN map}
 \Lambda_\sigma :   u|_{\p \Omega}\mapsto \nu \, \,\cdotp \sigma \nabla u|_{\p \Omega}.
\eeq
In the classical theory of the problem, the conductivity $\sigma$
is measurable and bounded uniformly from above and below by positive constants. We say that the conductivity $\sigma$ is \emph{isotropic} if it is a scalar function times the identity matrix, and otherwise \emph{anisotropic}.

We first discuss known results for isotropic conductivities. The inverse conductivity problem was proposed by Calder\'on \cite{C} in 1980. In dimensions $n \geq 3$, Sylvester and Uhlmann \cite{SU} proved unique identifiability of conductivities which are $C^2$ smooth, and Haberman and Tataru  \cite{HabermanTataru} proved uniqueness for $C^1$ conductivities. For earlier results for nonsmooth conductivities, see \cite{GLU1,PPU}. In two dimensions, the first global uniqueness result is due to Nachman \cite{N1}   for conductivities with two derivatives by the $\overline \p$ technique.
Astala and P\"aiv\"arinta  \cite{AP} established uniqueness for
$L^{\infty}$ conductivities which are bounded from above and below by positive constants.  See \cite{U_IP} for further references, and for numerical implementations see 
\cite{Siltanen1,Siltanen2,MuellerSiltanenBOOK,regulDbar}. 

In the case of anisotropic, i.e.~matrix-valued,
conductivities that
are uniformly bounded from above and below, there is a known obstruction to uniqueness in the Calder\'on problem \cite{KV}. If $\Phi:\overline\Omega\to
\overline\Omega$ is a diffeomorphism with $\Phi|_{\partial \Omega} = \mathrm{Id}$, it follows that 
$$
\Lambda_{\Phi_* \sigma} = \Lambda_{\sigma}
$$
where $\Phi_* \sigma$ is the pushforward of $\sigma$ under $\Phi$ ($D\Phi$ is the Jacobian matrix), 
$$
(\Phi_* \sigma)(y) = \frac{(D\Phi(x)) \sigma(x) (D\Phi(x))^t}{\abs{\det D\Phi(x)}}\Big|_{x=\Phi^{-1}(y)}.
$$
In two dimensions \cite{S, N1,LU,  ALP,IUY} it is
known that the Dirichlet-to-Neumann map $\Lambda_{\sigma}$ determines a regular conductivity tensor $\sigma$ up to such a
diffeomorphism. This means that one can obtain an image of the interior of $\Omega$ in
deformed coordinates. 
Thus, the inverse problem is not uniquely
solvable, but the nonuniqueness of the problem can be characterized. The same problem in dimensions $n \geq 3$ is open, see \cite{LeU, LU, LTU, GS, DKSaU, DKLSa} for partial results.

The
inverse conductivity problem
in the two dimensional case has also been studied with  
degenerate, that is, non-regular, conductivities. 
Such conductivities appear in physical models where the medium
varies continuously from a perfect conductor to a perfect insulator.
As an 
example, we may consider the case where the conductivity goes to
zero or to infinity near $\p D$ where $D\subset\Omega$ is a smooth open set.
The paper \cite{ALP2} characterizes the kinds of degeneracies for which one has nonuniqueness in the
inverse problem (the limit of visibility). The same paper also determines the cases where it is even possible 
to coat of an arbitrary object so that it appears like a homogeneous  body in all static boundary measurements (the limit of invisibility cloaking). 
Thus, the degree of nonuniqueness in the inverse conductivity equation
is quite well understood in two dimensions. Surprisingly, in the discrete problems that we study in
this paper a new type of nonuniqueness appears. 

\subsection{Equivalence of FEM models and resistor networks}

We will now state the first main theorems of this paper concerning the equivalence of FEM models and resistor networks. To keep this introduction brief we will not describe the discrete models in full detail, but rather refer to Sections \ref{sec_fem} and \ref{sec_resistor} for precise definitions.

For the FEM model, let $\Omega \subset \mR^2$ be the interior of a simple polygon and let $\mathscr{T}$ be a triangulation of $\overline{\Omega}$. We consider conductivity matrices from the class 
\begin{align*}
PC(\mathscr{T}) = \{ \sigma \in L^{\infty}(\Omega, \mR^{2 \times 2}) \,;\, & \sigma \text{ is a symmetric constant matrix } \\
 &\text{in the interior of each triangle} \},
\end{align*}
and also from the class $PC_+(\mathscr{T})$ which consists of those matrices in $PC(\mathscr{T})$ that are strictly positive definite in each triangle. 
Below, we define a {\it FEM model} to be the pair
 $(\mathscr{T},\sigma)$ where $\mathscr{T}$ is a
 triangulation of a simple polygon and
  $\sigma\in PC(\mathscr{T})$. 
The FEM model for the conductivity equation leads to an equation of the form \eqref{fem_equation_intro}, where the matrix on the left hand side will be denoted by $L(\sigma)$. If $L(\sigma)$ is invertible, we say that $0$ is not a Dirichlet eigenvalue of the FEM model (if $\sigma \in PC_+(\mathscr{T})$ this is always true). In this case, there are natural boundary measurements for the FEM model that are given by the Dirichlet-to-Neumann map 
$$
\Lambda_{\sigma}^{FEM} = \Lambda_{\sigma}^{FEM(\mathscr{T})}: \mR^B \to \mR^B
$$
where $B$ is the number of boundary vertices.

We next consider resistor networks. By definition, a resistor network is a graph $(V,E)$ together with a distinguished set $V_B \subset V$ a map $\gamma: E \to \mR$. Here $V = \{ x_1, \ldots, x_N \}$ is some finite set, called the set of vertices, $E \subset \{ \{x,y\} \,;\, x, y \in V\}$ is the set of edges, $V_B$ is the set of boundary vertices, and $\gamma(e)$ is thought of as the conductivity of a resistor that occupies the edge $e$. The Dirichlet problem for the resistor network is a matrix equation of the form \eqref{rn_equation_intro}, where the matrix on the left is denoted by $\tilde{L}(\gamma)$. If $\tilde{L}(\gamma)$ is invertible we say that $0$ is not a Dirichlet eigenvalue of the resistor network, and one has boundary measurements for the resistor network given by a Dirichlet-to-Neumann map 
$$
\Lambda_{\gamma}: \mR^B \to \mR^B.
$$
It turns out that there is a canonical way of going from a given FEM model to a resistor network having the same boundary measurements.

\begin{thm} \label{thm_main_fem_to_resistor}
Let $\Omega \subset \mR^2$ be the interior of a simple polygon, let $\mathscr{T}$ be a triangulation of $\overline{\Omega}$ having vertices $V$, boundary vertices $V_B$, and edges $E$, and let $\sigma \in PC_+(\mathscr{T})$. There is a unique function $\gamma$ such that the resistor network $R = (V,V_B,E,\gamma)$ satisfies $\tilde{L}(\gamma) = L(\sigma)$ and $\Lambda_{\gamma} = \Lambda_{\sigma}^{\mathrm{FEM}}$.
\end{thm}

Even if any FEM model with given triangulation naturally gives rise to a resistor network, there are many ways of going from the resistor network back to the FEM model. This may be described by going from the resistor network to a ''double resistor network'' where the resistor corresponding to each interior edge has been split in two parts. The splitting is determined by a real valued function $\alpha$ defined on the set of interior edges. It turns out that there is an essentially unique way of going from this double resistor network back to a FEM model with conductivity matrix $\sigma_{\alpha}$ that is constant on each triangle of the original triangulation and has the same Dirichlet-to-Neumann map as the original FEM model. The detailed construction of $\sigma_{\alpha}$ is given in Section \ref{sec_equivalence}. The conductivities that arise in this way are not always positive definite, but we will later also give conditions that guarantee positive definiteness.

\begin{thm}
\label{network fem equivalence}
Let $(V,E,V_B,\gamma)$ be a resistor network that corresponds to a triangulation $\mathscr{T}$, and assume that $0$ is not a Dirichlet eigenvalue. For any real valued function $\alpha$ defined on the set of interior edges, there is a symmetric matrix function $\sigma_{\alpha} \in PC(\mathscr{T})$ such that $L(\sigma_{\alpha}) = \tilde{L}(\gamma)$ and $\Lambda_{\sigma_{\alpha}}^{\text{FEM}} = \Lambda_{\gamma}$. Moreover, any $\sigma \in PC(\mathscr{T})$ that satisfies $L(\sigma) = \tilde{L}(\gamma)$ and $\Lambda_{\sigma}^{\text{FEM}} = \Lambda_{\gamma}$ is of the form $\sigma = \sigma_{\alpha}$ for some such function $\alpha$.
\end{thm}

We mention here that inverse problems for resistor networks have been extensively studied using graph theoretical methods \cite{BD,CM}. These studies have focused on two different types of problems. First, one can model a body with a fixed resistor network, e.g.\ a square lattice  corresponding to Euclidean coordinates or a circular lattice corresponding to polar coordinates  \cite{CdV1,CdV2,CIM,CMM,Ing}  and ask  whether one can reconstruct the resistors from boundary measurements. Second, one can consider arbitrary graphs with a subset of the nodes declared as the boundary nodes and use the fact that the resistor network can be modified by applying a so-called $\Delta-Y$ transformation (or the triangle-star transformation) in the graph without changing the boundary measurements. In this transformation the graph structure, in particular the number of the nodes, changes. In these studies one considers e.g.~the planarity of the graph and the characterization of equivalence classes of the graphs that appear the same in boundary measurements \cite{CdV1,CdV2,CIM}. Our paper falls closer to the first category of problems as we consider triangulations of a given domain which correspond to a fixed abstract graph.

\subsection{Nonuniqueness for FEM models.}
To consider nonuniqueness in the inverse problem for FEM models, let us start from
the  familiar nonuniqueness related to the diffeomorphism invariance of the conductivity equation $\nabla\,\cdotp\sigma \nabla u=0$. As pointed out above, the Dirichlet-to-Neumann maps for $\sigma$ and a pushforward conductivity $\Phi_* \sigma$ coincide if $\Phi: \overline{\Omega} \to \overline{\Omega}$ is a diffeomorphism fixing the boundary. If one approximates the conductivity equation by the FEM model and considers the corresponding resistor network, then the resistor network is an abstract graph which may be embedded in $\mR^2$ in several ways. These different embeddings (which may for instance move the interior vertices around) can be thought to correspond to the diffeomorphisms $\Phi$, and the use of resistor networks ''factors out'' the diffeomorphism. Thus the graph structure associated with the resistor network can be considered
as the analogue of the underlying Riemannian manifold structure associated with the 
anisotropic conductivity, cf.~\cite{LTU,LU,LeU}.

We will next formalize the analogue of diffeomorphism nonuniqueness. Two triangulations $\mathscr{T}$ and $\mathscr{T}'$ of $\overline{\Omega}$ are said to be isomorphic via a map $\Phi$ if $\Phi: \overline{\Omega} \to \overline{\Omega}$ is an orientation preserving homeomorphism that gives a $1$-$1$ correspondence between triangles of $\mathscr{T}$ and $\mathscr{T}'$ and is affine on each triangle. If $\sigma \in PC_+(\mathscr{T})$ and if $\Phi$ is an isomorphism of triangulations $\mathscr{T}$ and $\mathscr{T}'$, we define the pushforward of $\sigma$ via $\Phi$ by 
$$
\Phi_* \sigma(y) = \frac{(D\Phi(x)) \sigma(x) (D\Phi(x))^t}{\abs{\det D\Phi(x)}}\Big|_{x=\Phi^{-1}(y)},
$$
where $y$ is in the interior of some triangle of $\mathscr{T}'$ (see Fig.\ 1).

\begin{thm} \label{main_thm_diffeo_invariance}
Let $\mathscr{T}$ and $\mathscr{T}'$ be two triangulations of $\Omega \subset \mR^2$ that are isomorphic via some $\Phi: \overline{\Omega} \to \overline{\Omega}$ with $\Phi|_{\partial \Omega} = \mathrm{Id}$. If $\sigma \in PC_+(\mathscr{T})$, then we have 
$$
\Lambda_{\sigma}^{FEM(\mathscr{T})} = \Lambda_{\Phi_* \sigma}^{FEM(\mathscr{T}')}.
$$
\end{thm}

Thus diffeomorphism invariance in the FEM model looks basically the same as in the continuum model. However, there is another type of nonuniqueness in the discrete model which is not present in the continuous case. As mentioned above, any FEM model with given triangulation naturally gives rise to a resistor network, but there are many ways of going from the resistor network back to the FEM model. This involves going from the resistor network to a ''double resistor network'' where the resistor corresponding to each interior edge has been split in two parts. The different splittings will be parametrized by real valued functions $\beta$ on the set of interior edges, and these splittings account for the other type of nonuniqueness.

To make this precise, we consider the directed graphs associated with triangulations. Let $\mathscr{T}$ be a triangulation of $\Omega$ and let $\sigma \in PC_+(\mathscr{T})$. The non-directed graph associated with $\mathscr{T}$ is the pair $(V,E)$ where $V = \{ x_1, \ldots, x_N \}$ is the set of all vertices of triangles in $\mathscr{T}$ and $E$ is the set of non-directed edges. The directed graph associated with $\mathscr{T}$ is the pair $(V,D)$ where $V$ is the set of vertices as above, and $D$ is the set of all ordered pairs $(x_j, x_k)$ such that the line segment from $x_j$ to $x_k$ is a positively oriented edge of some triangle in $\mathscr{T}$. (The triangles of $\mathscr{T}$ and their boundaries have an orientation induced by the standard orientation of $\mR^2$.)

We write $E_B$ for the set of non-directed boundary edges (those edges that belong to only one triangle), and $E_I$ for the set of non-directed interior edges (those edges that belong to exactly two triangles). We assume that the elements of $V$ are labeled so that the set of boundary vertices (those vertices that lie on $\partial \Omega$) is $V_B = \{ x_1, \ldots, x_B \}$. Write $V_I = V \setminus V_B$ for the interior vertices.

 Let  $\beta: E_I \to \mR$ be a real valued function, and let $\tilde{\sigma}$ be the matrix function in $\Omega$ satisfying the following conditions:

\begin{enumerate}
\item[(i)] 
$\tilde{\sigma}$ is a constant matrix in the interior of each triangle in $\mathscr{T}$,
\item[(ii)]
if $e = \{x_j,x_k\} \in E_I$ and $j < k$, and if $T$ and $T'$ are the two triangles of $\mathscr{T}$ having both $x_j$ and $x_k$ as vertices such that $\partial T$ has the same orientation as the directed edge $d = (x_j,x_k)$ (so then $\partial T'$ has the opposite orientation from $d$), then 
\begin{align*}
\int_T \tilde{\sigma} \nabla v_j \cdot \nabla v_k \,dx &= \int_T \sigma \nabla v_j \cdot \nabla v_k \,dx + \beta(e), \\
\int_{T'} \tilde{\sigma} \nabla v_j \cdot \nabla v_k \,dx &= \int_{T'} \sigma \nabla v_j \cdot \nabla v_k \,dx - \beta(e),
\end{align*}

\item[(iii)]
if $e = \{x_j, x_k\} \in E_B$, and if $T$ is the unique triangle which has both $x_j$ and $x_k$ as vertices, then 
\begin{equation*}
\int_T \tilde{\sigma} \nabla v_j \cdot \nabla v_k \,dx = \int_T \sigma \nabla v_j \cdot \nabla v_k \,dx.
\end{equation*}
\end{enumerate}

It will turn out (see Section \ref{sec_equivalence}) that there is indeed a unique matrix function $\tilde{\sigma}$ in $\Omega$ that satisfies the above conditions, and we write $S_{\beta}(\sigma) = \tilde{\sigma}$. Clearly $S_{\beta}(\sigma) \in PC(\mathscr{T})$, and we also have $S_0(\sigma) = \sigma$. Note that here the conductivity $S_{\beta}(\sigma)$ may not be positive definite, even if $\sigma\in PC_+(\mathscr{T})$, but it is always such that $0$ is not a Dirichlet eigenvalue of the FEM problem and there is a well defined Dirichlet-to-Neumann map. The next result makes precise the second type of nonuniqueness that is present in the discrete model.

\begin{thm} \label{thm_main_nonuniqueness1}
Let $\mathscr{T}$ be a triangulation of $\Omega$ and let $\sigma \in PC_+(\mathscr{T})$. For any real valued function $\beta$ on $E_I$, there exists a unique conductivity matrix $\tilde{\sigma} = S_{\beta}(\sigma) \in PC(\mathscr{T})$ that satisfies (i)--(iii). One also has  
$$
\Lambda_{\sigma}^{FEM(\mathscr{T})} = \Lambda_{S_{\beta}(\sigma)}^{FEM(\mathscr{T})}.
$$
\end{thm}

We note that in the deformations considered in Theorem \ref{thm_main_nonuniqueness1},
the graph structure associated with the conductivities $\sigma$ and  $S_{\beta}(\sigma)$
is the same. Finally, it is instructive to state an ''infinitesimal'' version of the above two results.

\begin{thm} \label{thm_main_nonuniqueness2}
Let $\mathscr{T}$ be a triangulation of $\Omega$, let $w: V \to \mR^2$ be a map with $w|_{V_B} = 0$, and define $\Phi_t: V \to \mR^2$ for $t \in \mR$ by 
$$
\Phi_t(x_j) = x_j + t w(x_j), \quad x_j \in V.
$$
If $t$ is sufficiently small, let $\mathscr{T}_t$ be the triangulation of $\Omega$ that is isomorphic to $\mathscr{T}$ via $\Phi_t$. Let also $\beta$ be a map $E_I \to \mR$. Then 
$$
\Lambda_{\sigma}^{FEM(\mathscr{T})} = \Lambda_{(\Phi_t)_* S_{t \beta}(\sigma)}^{FEM(\mathscr{T}_t)}.
$$
\end{thm}

In this result, one may think that $\Phi_t$ moves around the vertices of the triangulation slightly, and $t \mapsto S_{t\beta}(\sigma)$ deforms the conductivity slightly according to the double resistor network picture  (see Fig.\ 2). We emphasize that if $\sigma$ is positive definite and $t>0$ is small enough, then also
the matrix $(\Phi_t)_* S_{t \beta}(\sigma)$ is positive definite. Roughly speaking, this shows that
the degree of the non-uniqueness in inverse problems for discrete FEM models,
even when the graph structure associated with the triangulation is fixed, is strictly larger than
than the one for the conductivity equation considered as a partial differential equation.

\subsection{Discrete non-uniqueness and cloaking}
It is a well known fact that the problem of determining an anisotropic conductivity is unique only up to pull-backs by diffeomorphisms. In the discrete setting, this corresponds to different embeddings of a planar graph into $\mR^2$.

\begin{center}
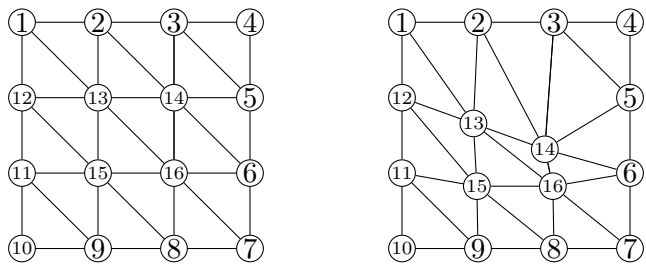
\begin{figure}[ht]
 \begin{tikzpicture}[scale = 0.5]
    \tikzstyle{every node}=[draw,circle,fill=white,minimum size=10pt,
                            inner sep=0pt]
    \draw (-5,-5) node (1) [] {1}
        -- ++(0:2.0cm) node (2) [] {2};
        
     \draw(2) 
        -- ++(0:2.0cm) node (3) [] {{3}};
      \draw(3) 
        -- ++(0:2.0cm) node (4) [] {4};
        
        \draw(4) 
        -- ++(270:2.0cm) node (5) [] {5}; 
         
        \draw(5) 
        -- ++(270:2.0cm) node (6) [] {6}; 

          \draw(6) 
        -- ++(270:2.0cm) node (7) [] {7}; 
        
           \draw(7) 
        -- ++(180:2.0cm) node (8) [] {8}; 
          \draw(8) 
        -- ++(180:2.0cm) node (9) [] {9}; 
  \draw(9) 
        -- ++(180:2.0cm) node (10) [] {{\tiny 10}}; 
        
          \draw(10) 
        -- ++(90:2.0cm) node (11) [] {{\tiny11}}; 

  \draw(11) 
        -- ++(90:2.0cm) node (12) [] {{\tiny 12}}; 
          \draw(12) --++(1);



  \draw(11) 
        -- ++(0:2.0cm) node (15) [] {{\tiny15}}; 
        
   \draw(15) 
        -- ++(0:2.0cm) node (16) [] {{\tiny16}}; 
        
    \draw(16) --++ (6);

  \draw(12) 
        -- ++(0:2.0cm) node (13) [] {{\tiny13}}; 
        
   \draw(13) 
        -- ++(0:2.0cm) node (14) [] {{\tiny 14}}; 
        
    \draw(14) --++ (5);



\draw(2) --++ (13);
\draw(3) --++ (14);
\draw(13) --++ (15);
\draw(15) --++ (9);
\draw(3) --++ (14);
\draw(14) --++ (16);
\draw(16)--++(8);



\draw(1) --++(13);
\draw(13) --++ (16);
\draw(16) --++(7);
\draw(12) --++(15);
\draw(15) --++ (8);
\draw(11) --++(9);
\draw(2) --++(14);
\draw(14) --++(16);
\draw(3) --++ (5);
\draw(14) --++(6);


    \draw (5,-5) node (101) [] {1}
        -- ++(0:2.0cm) node (102) [] {2};
        
     \draw(102) 
        -- ++(0:2.0cm) node (103) [] {3};
      \draw(103) 
        -- ++(0:2.0cm) node (104) [] {4};
        
        \draw(104) 
        -- ++(270:2.0cm) node (105) [] {5}; 
         
        \draw(105) 
        -- ++(270:2.0cm) node (106) [] {6}; 

          \draw(106) 
        -- ++(270:2.0cm) node (107) [] {7}; 
        
           \draw(107) 
        -- ++(180:2.0cm) node (108) [] {8}; 
          \draw(108) 
        -- ++(180:2.0cm) node (109) [] {9}; 
  \draw(109) 
        -- ++(180:2.0cm) node (110) [] {{\tiny10}}; 
        
          \draw(110) 
        -- ++(90:2.0cm) node (111) [] {{\tiny11}}; 

  \draw(111) 
        -- ++(90:2.0cm) node (112) [] {{\tiny12}}; 
          \draw(112) --++(101);



  \draw(111) 
        -- ++(350:2.0cm) node (115) [] {{\tiny15}}; 
        
   \draw(115) 
        -- ++(0:2.0cm) node (116) [] {{\tiny16}}; 
        
    \draw(116) --++ (106);

  \draw(112) 
        -- ++(340:2.0cm) node (113) [] {{\tiny13}}; 
        
   \draw(113) 
        -- ++(340:2.0cm) node (114) [] {{\tiny14}}; 
        
    \draw(114) --++ (105);



\draw(102) --++ (113);
\draw(103) --++ (114);
\draw(113) --++ (115);
\draw(115) --++ (109);
\draw(103) --++ (114);
\draw(114) --++ (116);
\draw(116)--++(108);



\draw(101) --++(113);
\draw(113) --++ (116);
\draw(116) --++(107);
\draw(112) --++(115);
\draw(115) --++ (108);
\draw(111) --++(109);
\draw(102) --++(114);
\draw(114) --++(116);
\draw(103) --++ (105);
\draw(114) --++(106);
\end{tikzpicture}
\caption{Two different embeddings of a planar graph.}
\end{figure}
\end{center}
When considering a FEM model and its corresponding resistor networks there is an additional non-uniqueness. Namely, the piecewise constant conductivity on each triangle can be distributed differently along the adjacent edges. 
\begin{center}
\begin{figure}[ht]
 \begin{tikzpicture}[scale = 0.5]
    \tikzstyle{every node}=[draw,circle,fill=white,minimum size=4pt,
                            inner sep=0pt]
    \draw (-15,0) node (1) [] {1}
        -- ++(0:6.0cm) node (2) [] {2};
  \draw(2) --++ (270: 6cm) node(3) [] {{3}}; 
  \draw(3) --++ (180: 6cm) node(4) [] {4}; 
  \draw(4) --++ (1);
  \draw(4) --++ (45: 4.25cm) node(5) [] {5};
  \draw(1) --++(5) ;
  \draw(2) --++(5);
  \draw(3) --++(5);
 \tikzstyle{every node}=[draw,circle,fill=white,minimum size=4pt,
                            inner sep=0pt]
    \draw (-7.5,0) node (11) [] {1}
        -- ++(0:6.0cm) node (12) [] {2};
  \draw(12) --++ (270: 6cm) node(13) [] {3}; 
  \draw(13) --++ (180: 6cm) node(14) [] {4}; 
  \draw(14) --++ (11);
 \draw(-4.5,-3) node(15) []{5};
  
  \path[every node/.style={font=\sffamily\small}]
    (11) edge [bend right] node[left] {} (15)
   (11) edge [bend left] node[left] {} (15)
   (12) edge [bend right] node[left] {} (15)
   (12) edge [bend left] node[left] {} (15)
    (13) edge [bend right] node[left] {} (15)
   (13) edge [bend left] node[left] {} (15)
   (14) edge [bend right] node[left] {} (15)
   (14) edge [bend left] node[left] {} (15);
          \draw (0,0) node (21) [] {1}
        -- ++(0:6.0cm) node (22) [] {2};
  \draw(22) --++ (270: 6cm) node(23) [] {3}; 
  \draw(23) --++ (180: 6cm) node(24) [] {4}; 
  \draw(24) --++ (21);
 \draw(3,-3) node(25) []{5};
  
  \path[every node/.style={font=\sffamily\small}]
    (21) edge [bend right] node[left] {} (25)
   (21) edge [bend left] node[left] {} (25)
   (22) edge [bend right, red] node[left] {} (25)
   (22) edge [bend left, blue] node[left] {} (25)
    (23) edge [bend right, red] node[left] {} (25)
   (23) edge [bend left, blue] node[left] {} (25)
   (24) edge [bend right] node[left] {} (25)
   (24) edge [bend left] node[left] {} (25);
    
\end{tikzpicture}
\caption{From left to right: The triangulation of a domain; corresponding resistors network with even splitting along the edges; corresponding resistor network with uneven splitting along the edges.}

\end{figure}
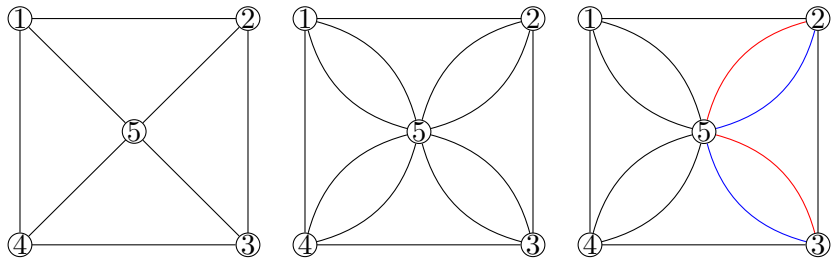
\end{center}
In the last section of the paper we will also discuss discrete cloaking layers where any resistor network $\Omega_*$ can be embedded in a cloaking resistor network without being detected and the corresponding FEM models. 

\subsection*{Structure of paper}
Section \ref{sec_intro} is the introduction and contains the statements of the main results. In Section \ref{sec_fem} we describe the FEM model for the conductivity equation and also prove Theorem \ref{main_thm_diffeo_invariance}. Resistor networks are discussed in Section \ref{sec_resistor}. The correspondence between FEM models and resistor networks is explained in Section \ref{sec_equivalence}, where also Theorems \ref{thm_main_fem_to_resistor}, \ref{network fem equivalence}, \ref{thm_main_nonuniqueness1} and \ref{thm_main_nonuniqueness2} are proved. The final section describes invisibility cloaking constructions 
for FEM models and resistor networks.

\subsection*{Acknowledgements}
M.L., M.S., and L.T. are supported in part by the Academy of Finland (Finnish Centre of Excellence in Inverse Problems Research), and M.S.~is also supported by an ERC Starting Grant. L.T~is partly supported by The Swedish Research Council (Vetenskapsr\aa det). The authors would like to thank the organizers of the Institut Mittag-Leffler program on Inverse Problems where part of this work was done.

\section{FEM model} \label{sec_fem}

\subsection{Dirichlet problem for the FEM model}
A \emph{triangulation} is a finite set $\mathscr{T} = \{ T_1, \ldots, T_R \}$ where each $T_j$ is a closed triangle in $\mR^2$, and two triangles $T_j$ and $T_k$ for $j \neq k$ are either disjoint or their intersection is one vertex or one full edge of each triangle. The \emph{domain} of $\mathscr{T}$ is the set 
$$
\Omega = \text{int}\left( \bigcup_{j=1}^R T_j \right),
$$
and we have $\overline{\Omega} = \bigcup_{j=1}^R T_j$. Two triangulations $\mathscr{T} = \{ T_1, \ldots, T_R \}$ and $\mathscr{T}' = \{ T_1', \ldots, T_{R'}' \}$ with domains $\Omega$ and $\Omega'$ are said to be \emph{isomorphic} if $R = R'$ and there is an orientation preserving homeomorphism $\Phi: \overline{\Omega} \to \overline{\Omega}'$ that is affine on each $T_j$ and maps each $T_j$ onto some $T_{j'}'$ (the map $\Phi$ is called an isomorphism).

We denote the set of all vertices of triangles $T\in \mathscr{T}$ by
$V = \{ x_1, \ldots, x_N \}$. We assume that the points are labeled so that  $V_B = \{ x_1, \ldots, x_B \}=V\cap \p \Omega$, where $B\leq N$, is the set of vertices that lie on $\partial \Omega$.

Given a triangulation $\mathscr{T}$ of a domain $\Omega$, we will consider conductivity matrices from the class 
\begin{align*}
PC(\mathscr{T}) = \{ \sigma \in L^{\infty}(\Omega, \mR^{2 \times 2}) \,;\, & \sigma \text{ is a symmetric constant matrix } \\
 &\text{in the interior of each triangle} \}.
\end{align*}
We also write $PC_+(\mathscr{T})$ for the subspace of matrices that are positive definite on each triangle. 

Below, we say that  $v:\overline \Omega\to \mR$ is a piecewise affine function
(associated with the triangulation $\mathscr{T}$) if $v$
 is a continuous function on $\overline\Omega$ and the restriction
$v_j|_{T_k}$ is an 
affine function for each triangle $T_k\in \mathscr{T}$ . The set of piecewise affine functions is
denoted by $A( \Omega) =A_{\mathscr{T}}( \Omega)$. 
 Moreover, we denote by $v_k$, $k=1,2,\dots,N$, the specific basis of 
the space $A( \Omega)$  for which $v_k$ has value $1$ at the vertex $x_k$
and vanishes at all other vertices $x_j$, $j\not= k.$
Also, write  
$$
A(\partial \Omega) = \{ v|_{\partial \Omega} \,;\, v \in A(\Omega) \} = \{ \sum_{j=1}^B c_j v_j|_{\p {\Omega}}\, ; \ c_j \in \mR \}
$$
for the space of piecewise affine functions (associated with the triangulation $\mathscr{T}$) on the boundary.

Next we will formulate Dirichlet problem for the FEM model.
To motivate it,
consider the Dirichlet problem for the anisotropic conductivity equation 
\begin{equation} \label{conductivity_dirichlet}
\nabla \cdot (\sigma (x)\nabla u(x)) = 0 \text{ in } \Omega, \quad u|_{\partial \Omega} = f,
\end{equation}
where $\Omega \subset \mR^2$ is a bounded domain, and $\sigma \in L^{\infty}(\Omega;\mR^{2 \times 2})$ is a symmetric positive definite matrix. Weak solutions are functions in $H^1(\Omega)$ with $u|_{\partial \Omega} = f$ satisfying 
\begin{equation*}
\int_{\Omega} \sigma(x) \nabla u(x) \cdot \nabla \varphi(x) \,dx = 0 \quad \text{ for all } \varphi \in H^1_0(\Omega).
\end{equation*}

Let $\Omega$ be the interior of a simple polygon, and fix a triangulation of $\closure{\Omega}$ with vertices $\{x_1,\ldots,x_N\}$ of which $\{x_1,\ldots,x_B\}$ lie on $\partial \Omega$. Let $A(\Omega)$ be the set of continuous functions $\closure{\Omega} \to \mR^2$ which are affine on each triangle. 
Recall that $v_k(x)$ is the unique function in $A(\Omega)$ which is $1$ at $x_k$ and $0$ at all the other vertices and $A(\Omega) = \{ \sum_{k=1}^N c_k v_k(x) \,;\, c_k \in \mR \}$. Note that $A(\Omega) \subset H^1(\Omega)$ and the gradient of a function in $A(\Omega)$ is constant in the interior of each triangle. If $f\in A(\p \Omega)$, the FEM approximation to \eqref{conductivity_dirichlet} asks to find a function $u \in A(\Omega)$ such that 
\begin{equation} \label{fem_dp}
\left\{ \begin{array}{r@{\hspace{3pt}}ll}
\int_{\Omega} \sigma(x) \nabla u(x)  \cdot \nabla v_j(x)  \,dx &= 0,  & \quad j=B+1,\ldots,N, \\[5pt]
u(x_j) &= f(x_j), & \quad j=1,\ldots,B.
\end{array} \right.
\end{equation}
We refer to \cite{Ciarlet} for more information on FEM approximations.

\subsection{FEM model as a matrix equation}
Let $f\in A(\p \Omega)$   be a given piecewise affine 
function on the boundary and let $u\in A(\Omega)$ be the solution of equation (\ref{fem_dp}).
Writing $f(x_j) = f_j$ and $u(x) = \sum_{k=1}^N u_k v_k(x)$, (\ref{fem_dp}) is equivalent with the matrix equation 
\begin{equation} \label{fem_equation}
L u = \left( \begin{array}{cc} f \\ 0 \end{array} \right)
\end{equation}
where $u = (u_1 \cdots u_N)^t$ and $f = (f_1 \cdots f_B)^t$, and $L = (l_{jk})_{j,k=1}^N$ is the matrix given by 
\begin{equation} \label{L-matrix}
\begin{array}{l@{\hspace{20pt}}l}
l_{jk} = \delta_{jk},  & 1 \leq j \leq B, \\[5pt]
l_{jk} = \int_{\Omega} \sigma \nabla v_j \cdot \nabla v_k \,dx, & B+1 \leq j \leq N.
\end{array}
\end{equation}
We write $L = L(\sigma) = (l_{jk}(\sigma))_{j,k=1}^N$ for the matrix above. Next we record some well-known properties of the matrix $L$.

\begin{lemma}
If $\sigma$ is a symmetric matrix function in $L^{\infty}(\Omega ; \mR^{2 \times 2})$, then $L(\sigma)$ has the form 
\begin{equation*}
L = \left( \begin{array}{cc} I_{B \times B} & 0 \\ L' & L'' \end{array} \right)
\end{equation*}
where $L''$ is symmetric. The rows of $\left( \begin{array} {cc} L' & L''  \end{array} \right)$ add up to zero:
\begin{equation*}
\sum_{k=1}^N l_{jk} = 0, \quad j \geq B+1.
\end{equation*}
If $\sigma$ is positive definite (in the sense that $\sigma(x) \xi \cdot \xi \geq c\abs{\xi}^2$ for a.e.~$x \in \Omega$ and for all $\xi \in \mR^2$ where $c>0$), then $L$ is positive definite (thus invertible) and $l_{jj} > 0$ for $j \geq B+1$.
\end{lemma}
\begin{proof}
Clearly $L$ has the above form, and if $j \geq B+1$ then 
\begin{equation*}
\sum_{k=1}^N l_{jk} = \int_{\Omega} \sigma \nabla v_j \cdot \nabla (\sum_{k=1}^N v_k) \,dx = 0,
\end{equation*}
since $\sum_{k=1}^N v_k \equiv 1$ in $\Omega$. Assume now that $\sigma$ is positive definite. Then $L''$ is positive definite since for $z =( z_{B+1} \cdots z_N )^t \in \mR^{N-B} \setminus \{0\}$, 
$$
L'' z \cdot z = \sum_{j,k=B+1}^N l_{jk} z_j z_k = \int_{\Omega} \sigma \nabla w \cdot \nabla w \geq c \int_{\Omega} \abs{\nabla w}^2 > 0
$$
where $w = \sum_{k \geq B+1} z_k v_k$. Thus also $L$ is positive definite and invertible, and for $j \geq B+1$ we have 
$$
l_{jj} = \int_{\Omega} \sigma \nabla v_j \cdot \nabla v_j \,dx \geq c \int_{\Omega} \abs{\nabla v_j}^2 \,dx > 0.
$$
\end{proof}

It will be useful to consider the FEM approximation for conductivities which may not be positive definite. If $\sigma \in L^{\infty}(\Omega ; \mR^{2 \times 2})$ is a symmetric matrix, we say that $0$ is not an eigenvalue of the FEM problem if the matrix $L = L(\sigma)$ is invertible. In this case, for any $f = (f_1 \cdots f_B)^t$ the FEM problem \eqref{fem_equation} has a unique solution $u = L^{-1} f$, and one has a well defined Dirichlet-to-Neumann map $\Lambda_{\sigma}^{\mathrm{FEM}}$ defined below.

For the relation to resistor networks below, one might ask for the condition $l_{jk} \leq 0$ for $j \neq k$ (this would ensure that no resistor has negative conductivity). This is not true in general, but in the following case the non-negativity of elements of the matrix $L$ can be guaranteed.

\begin{lemma}
If $\sigma(x)$ is isotropic, i.e.~a scalar function multiplied with the identity matrix, and if the triangulation of $\closure{\Omega}$ is such that no triangle has an angle $> \pi/2$, then $l_{jk} \leq 0$ for $j \neq k$.
\end{lemma}
\begin{proof}
It is enough to consider $l_{jk}$ for $j \geq B+1$ and $k \neq j$. One has $l_{jk} = 0$ unless $x_j$ and $x_k$ are neighboring vertices, and in that case there are exactly two triangles $T_1$ and $T_2$ which have both $x_j$ and $x_k$ as vertices. Since no angle is $> \pi/2$, elementary geometric considerations imply that in these two triangles $\nabla v_j \cdot \nabla v_k \leq 0$. Thus also $\sigma \nabla v_j \cdot \nabla v_k \leq 0$ in these triangles.
\end{proof}

\subsection{Dirichlet-to-Neumann map for FEM model}
We will next define a Dirichlet-to-Neumann map $\Lambda_{\sigma}^{\mathrm{FEM}}$ related to the FEM model. Assume that $0$ is not a Dirichlet eigenvalue for the FEM model, so (\ref{fem_dp}) has a unique
solution with all $f\in A(\p \Omega)$. We define the Dirichlet-to-Neumann map  $\Lambda_{\sigma}^{\mathrm{FEM}}:
A(\partial \Omega)\to A(\partial \Omega)$ to be the unique map  given by the pairing
\begin{equation*}
\langle \Lambda_{\sigma}^{\mathrm{FEM}} f, g \rangle_{A(\partial \Omega)} = \int_{\Omega} \sigma \nabla u \cdot \nabla w \,dx , \quad \hbox{for all } f, g \in A(\partial \Omega),
\end{equation*}
where $u \in A(\Omega)$ is the solution of \eqref{fem_dp}, and $w$ is any function in $A(\Omega)$ with $w|_{\partial \Omega} = g$. Here we have written 
\begin{equation*}
\langle f, g \rangle_{A(\partial \Omega)} = \sum_{j=1}^B f(x_j) g(x_j).
\end{equation*}
Writing $u(x) = \sum_{j=1}^N u_j v_j(x)$, $u_j = u(x_j) \in \mR$, we obtain an explicit formula for the Dirichlet-to-Neumann map:
\begin{equation}\label{FEM DN map}
\Lambda_{\sigma}^{\mathrm{FEM}} f(x_j) = \sum_{k=1}^N \Big( \int_{\Omega} \sigma(x) \nabla v_j(x) \cdot \nabla v_k(x) \,dx \Big) u_k.
\end{equation}

\subsection{Diffeomorphism invariance}

If $\sigma \in PC(\mathscr{T})$ and if $\Phi$ is an isomorphism of two triangulations $\mathscr{T}$ and $\mathscr{T}'$, we define the pushforward of $\sigma$ via $\Phi$ by 
$$
\Phi_* \sigma(y) = \frac{(D\Phi(x)) \sigma (x)(D\Phi(x))^t}{\abs{\det D\Phi(x)}}\Big|_{x=\Phi^{-1}(y)},
$$
where $y$ is in the interior of some triangle of $\mathscr{T}'$. Since $\Phi$ is affine on each triangle, it follows that $\Phi_* \sigma \in PC(\mathscr{T}')$. Moreover, if $\sigma \in PC_+(\mathscr{T})$ then $\Phi_* \sigma \in PC_+(\mathscr{T}')$.


\begin{lemma}
Let $\mathscr{T}$ and $\mathscr{T}'$ be two triangulations of $\Omega \subset \mR^2$ that are isomorphic via some $\Phi: \overline{\Omega} \to \overline{\Omega}$ with $\Phi|_{\partial \Omega} = \mathrm{Id}$, and label the vertices of the triangulations so that $x_j' = \Phi(x_j)$. If $\sigma \in PC(\mathscr{T})$ and if $0$ is not a Dirichlet eigenvalue for the FEM problem with conductivity $\sigma$, then $0$ is not a Dirichlet eigenvalue for the FEM problem with conductivity $\Phi_* \sigma$ and we have 
$$
L(\sigma) = L(\Phi_* \sigma), \qquad \Lambda_{\sigma}^{FEM(\mathscr{T})} = \Lambda_{\Phi_* \sigma}^{FEM(\mathscr{T}')}.
$$
\end{lemma}
\begin{proof}
By definition $l_{jk}(\sigma) = l_{jk}(\Phi_* \sigma) = \delta_{jk}$ for $j \leq B$. If $\{ v_1, \ldots, v_N \}$ and $\{ v_1', \ldots, v_N' \}$ are the basis functions of $A_{\mathscr{T}}(\Omega)$ and $A_{\mathscr{T}'}(\Omega)$, respectively, then 
$$
v_k'(y) = v_k(\Phi^{-1}(y)), \quad y \in \overline{\Omega},
$$
and therefore (writing $\nabla$ as a column vector) 
$$
\nabla v_k'(y) = D\Phi^{-1}(y)^t \nabla v_k(\Phi^{-1}(y)) = D\Phi(\Phi^{-1}(y))^{-t} \nabla v_k(\Phi^{-1}(y)).
$$
If $1 \leq j,k \leq N$, the change of variables $y = \Phi(x)$ implies 
$$
l_{jk}(\Phi_* \sigma) = \int_{\Omega} \Phi_* \sigma(y) \nabla v_j'(y) \cdot \nabla v_k'(y) \,dy = \int_{\Omega} \sigma \nabla v_j \cdot \nabla v_k \,dx = l_{jk}(\sigma).
$$
The result follows.
\end{proof}

\begin{proof}[Proof of Theorem \ref{main_thm_diffeo_invariance}]
This is just a special case of the previous lemma where $\sigma \in PC_+(\mathscr{T})$.
\end{proof}

\section{Resistor networks} \label{sec_resistor}

To properly state the equivalence of the FEM model and resistor networks, we need some concepts from graph theory (see \cite{Diestel}). By a graph we mean a set $(V,E)$ where $V$ (vertices) is a finite set, and $E$ (edges) is a set of unordered pairs of elements in $V$. A path from $x \in V$ to $y \in V$ is a sequence of vertices $(x_0,x_1,\ldots,x_k)$ where $x_0 = x$, $x_k = y$, and $\{x_j,x_{j+1}\} \in E$ for $0 \leq j \leq k-1$. We only consider graphs which are connected (any two vertices can be joined by a path) and simple (no self-loops, i.e.~$E$ has no elements of the form $\{x,x\}$).

\begin{definition}
A \emph{resistor network} is a graph $(V,E)$ together with a distinguished set $V_B \subset V$ and a function $\gamma: E \to \mR$.
\end{definition}

We will label elements of $V$ so that $V_B = \{x_1,\ldots,x_B\}$ (the set of boundary vertices) and $V \smallsetminus V_B = \{x_{B+1},\ldots,x_N\}$ (the set of interior vertices). If $e = \{x_j,x_k\} \in E$ we write $\gamma_{jk} := \gamma(e)$ (the conductivity of a resistor located on the edge $e$), with the convention that $\gamma_{jk} = 0$ if $\{x_j,x_k\} \notin E$. Given a function $f: V_B \to \mR$, we would like to find a function $u: V \to \mR$ satisfying 
\begin{equation} \label{resistor_dp}
\left\{ \begin{array}{r@{\hspace{3pt}}ll}
\sum_{x_k \in \mathcal{N}(x_j)} \gamma_{jk}[u(x_j) - u(x_k)] &= 0,  & \quad j=B+1,\ldots,N, \\[5pt]
u(x_j) &= f(x_j), & \quad j=1,\ldots,B.
\end{array} \right.
\end{equation}

Here $\mathcal{N}(x) = \{ y \in V \,;\, \{ x, y \} \in E \}$ is the set of neighboring vertices of $x$. If the graph is such that $(V,E)$ can be drawn as a straight-line graph in the plane, by Kirchhoff's law the function $u$ corresponds to a voltage potential in the resistor network induced by the voltage $f$ applied at the boundary vertices.

Writing $f(x_j) = f_j$ and $u(x_j) = u_j$, the problem \eqref{resistor_dp} corresponds to the matrix equation 
\begin{equation} \label{rn_equation}
\tilde{L} u = \left( \begin{array}{cc} f \\ 0 \end{array} \right)
\end{equation}
where $u = (u_1 \cdots u_N)^t$ and $f = (f_1 \cdots f_B)^t$, and $\tilde{L} = (\tilde{l}_{jk})_{j,k=1}^N$ with 
\begin{equation*}
\begin{array}{l@{\hspace{20pt}}l}
\tilde{l}_{jk} = \delta_{jk},  & j \leq B, \\[5pt]
\tilde{l}_{jj} = \sum_{x_k \in \mathcal{N}(x_j)} \gamma_{jk},  & j \geq B+1, \\[5pt]
\tilde{l}_{jk} = -\gamma_{jk}, & j \geq B+1, \ k \neq j.
\end{array}
\end{equation*}
It is clear that the matrix $\tilde{L}$ has the form 
\begin{equation*}
\tilde{L} = \left( \begin{array}{cc} I_{B \times B} & 0 \\ \tilde{L}' & \tilde{L}'' \end{array} \right)
\end{equation*}
where $\tilde{L}''$ is symmetric, and for $j \geq B+1$ one has 
\begin{equation*}
\sum_{k=1}^N \tilde{l}_{jk} = 0.
\end{equation*}
Further, if $\gamma(e) > 0$ for all $e \in E$ then $\tilde{l}_{jj} > 0$ for all $j$.

We next give a condition for the problem \eqref{resistor_dp}, equivalently \eqref{rn_equation}, to have a unique solution for any boundary value $f$. This holds if all resistors have positive conductivity and if the graph $(V',E')$ where $V' = \{x_{B+1},\ldots,x_N\}$ and $E' = \{\{x_j,x_k\} \in V' \times V' \,;\, \tilde{l}_{jk} \neq 0 \text{ and } j \neq k\}$ is connected (here $(V',E')$ is the graph obtained from $(V,E)$ by removing all boundary vertices and all edges whose end point is a boundary vertex).

\begin{lemma}
Assume that $\gamma(e) > 0$ for all edges $e \in E$ and that $(V',E')$ is connected. Then $\tilde{L}$ is invertible, and further $\tilde{L}''$ and also $\tilde{L}$ are positive definite.
\end{lemma}
\begin{proof}
Since the conductivities are positive, $\tilde{L}''$ is diagonally dominant:
\begin{equation} \label{diagonal_dominance}
\tilde{l}_{jj} \geq \sum_{\stackrel{k=B+1}{k \neq j}}^N \abs{\tilde{l}_{jk}}, \qquad j \geq B+1.
\end{equation}
This implies that $\tilde{L}''$ is positive semidefinite, as a symmetric diagonally dominant matrix with nonnegative diagonal entries (this is a consequence of Gershgorin's circle theorem). For invertibility of $\tilde{L}''$ we note that if there is strict inequality in \eqref{diagonal_dominance} for some $j \geq B+1$ and if the graph $(V',E')$ is connected, then results based on Gershgorin disks (see \cite[Section 6.2]{HJ}) show that $\tilde{L}''$ is indeed invertible and thus also positive definite. But strict inequality in \eqref{diagonal_dominance} holds for some $j \geq B+1$ since otherwise no interior vertex would have a neighboring boundary vertex, so the set of interior vertices would be empty by connectedness.
\end{proof}

If $\tilde{L}$ is invertible, we say that $0$ is not an eigenvalue of \eqref{resistor_dp}. In this case it is possible to define the Dirichlet-to-Neumann map of the resistor network as the operator acting on functions $f: V_B \to \mR$ by 
\begin{equation*}
\Lambda_{\gamma} f(x_j) = \sum_{x_k \in \mathcal{N}(x_j)} \gamma_{jk} [u(x_j)-u(x_k)], \quad j=1,\ldots,B,
\end{equation*}
where $u$ is the unique solution of \eqref{resistor_dp}.

\section{Correspondence of FEM models and resistor networks} \label{sec_equivalence}

We wish to relate the FEM model to resistor networks. To do this we need to consider resistor networks which can be drawn as the triangulation of a simple polygon in $\mR^2$. A (planar straight-line) \emph{drawing} of a graph $(V,E)$ is a pair of maps $\varphi: V \to \mR^2$ and $\psi: E \times [0,1] \to \mR^2$, where $\varphi$ is injective and $\psi$ satisfies the following conditions:
\begin{enumerate}
\item[(1)]
if $e = \{x,y\} \in E$ then $\psi_e = \psi(e,\,\cdot\,)$ is a line segment with end points $\varphi(x)$ and $\varphi(y)$,
\item[(2)]
$\psi_e([0,1])$ does not intersect $\varphi(V)$ except at the end points,
\item[(3)]
for any $e, \tilde{e} \in E$ with $e \neq \tilde{e}$ the open line segments $\psi_e((0,1))$ and $\psi_{\tilde{e}}((0,1))$ do not intersect.
\end{enumerate}
The components of $\mR^2 \smallsetminus \psi(E \times [0,1])$ are called the faces of the graph (there is exactly one unbounded face). The number of faces is independent of the embedding by Euler's formula. A graph which admits a drawing is called planar. Planarity can be checked directly from the abstract graph (for instance Kuratowski's condition). Also, there is no loss of generality in using straight-line embeddings by Fary's theorem. See \cite{Diestel} for these facts.

\begin{definition}
A \emph{triangular resistor network} is a resistor network $(V,E)$ such that
\begin{enumerate}
\item[(1)]
there is a drawing of $(V,E)$ for which every face (except the unbounded one) is a triangle, 
\item[(2)]
the interior of the complement of the unbounded face, which we denote by $\Omega$ (the \emph{domain}), is such that $\closure{\Omega}$ is connected,
\item[(3)]
$\varphi(V_B) = \varphi(V) \cap \partial \Omega$.
\end{enumerate}
\end{definition}

\begin{remark}
There is a notion of triangulation in graph theory which only uses the abstract graph, but this seems to be different from what we need here.
\end{remark}

If $\sigma$ is an anisotropic conductivity in a triangulated set $\closure{\Omega}$, we denote by $L(\sigma)$ the matrix in \eqref{fem_equation}. Also, if $(V,V_B,E,\gamma)$ is a resistor network we denote by $\tilde{L}(\gamma)$ the matrix in \eqref{rn_equation}. It is easy to see that anisotropic conductivities give rise to resistor networks in this way.

\begin{thm} \label{thm_main_fem_to_resistor_general}
Let $\Omega \subset \mR^2$ be the interior of a simple polygon, and let $\sigma \in L^{\infty}(\Omega; \mR^{2\times 2})$ be a symmetric matrix function such that $0$ is not an eigenvalue of \eqref{fem_dp}. If a triangulation of $\closure{\Omega}$ with vertices $V$, boundary vertices $V_B = V \cap \partial \Omega$, and edges $E$ is given, then there is a unique function $\gamma$ such that the triangular resistor network $R = (V,V_B,E,\gamma)$ satisfies $\tilde{L}(\gamma) = L(\sigma)$ and $\Lambda_{\gamma} = \Lambda_{\sigma}^{\mathrm{FEM}}$.
\end{thm}
\begin{proof}
With the notation in Section \ref{sec_fem}, we choose 
$$
\gamma_{jk} = -\int_{\Omega} \sigma \nabla v_j \cdot \nabla v_k \,dx, \qquad j \geq B+1, \ \ k \neq j.
$$
Note that $\int_{\Omega} \sigma \nabla v_j \cdot \nabla v_k \,dx = 0$ if $x_k \notin \mathcal{N}(x_j)$. Using the fact that 
$$
\sum_{k=1}^N \int_{\Omega} \sigma \nabla v_j \cdot \nabla v_k \,dx = \int_{\Omega} \sigma \nabla v_j \cdot \nabla(1) \,dx = 0,
$$
we have $\tilde{l}_{jj} = \sum_{x_k \in \mathcal{N}(x_j)} \gamma_{jk} = \int_{\Omega} \sigma \abs{\nabla v_j}^2 \,dx$ and thus $\tilde{L}(\gamma) = L(\sigma)$. For the equality of Dirichlet-to-Neumann maps we also choose 
$$
\gamma_{jk} = -\int_{\Omega} \sigma \nabla v_j \cdot \nabla v_k \,dx, \quad j =1,\ldots,B, \ \ k \neq j.
$$
This implies $\sum_{x_k \in \mathcal{N}(x_j)} \gamma_{jk} = \int_{\Omega} \sigma \abs{\nabla v_j}^2 \,dx$ as above, so $\Lambda_{\gamma} = \Lambda_{\sigma}^{\mathrm{FEM}}$.
\end{proof}

\begin{proof}[Proof of Theorem \ref{thm_main_fem_to_resistor}]
This is a special case of the previous theorem when $\sigma \in PC_+(\mathscr{T})$.
\end{proof}

\begin{remark}
If $\sigma$ is a positive definite matrix then $\tilde{L}(\gamma)$ is positive definite and the diagonal entries $\tilde{l}_{jj}$ are positive, but apparently some of the conductivities $\gamma_{jk}$ may be negative unless additional restrictions are imposed.
\end{remark}

In the converse direction there are many conductivities $\sigma$ corresponding to a given resistor network. However, if one restricts to conductivities which are constant inside each triangle, it is possible to characterize all such $\sigma$ in terms of directed resistor networks. To describe the characterization, we let $(V,E)$ be a resistor network and consider a splitting into two parts of all the resistors associated with interior edges of $E$. Informally, this is done as follows:

\begin{enumerate}
\item 
Let $E_I = \{ e \in E \,;\, \text{at least one endpoint of $e$ is an interior vertex} \}$.
\item 
For each $e = \{x_j,x_k\} \in E_I$ choose an ordered pair $e' = (x_j,x_k)$, and let $D_I = \{ e' \,;\, e \in E_I \}$ be the corresponding set of directed edges.
\item 
Let $\alpha$ be a function $D_I \to \mR$.
\item
For any edge $e \in E_I$ with corresponding directed edge $e'$, we split the resistor $\gamma(e)$ associated with $e$ in two resistors, one on the ''left'' of $e'$ and one on the ''right'' of $e'$, such that the resistor on the left has conductivity $\frac{1}{2} \gamma(e) + \alpha(e')$ and the one on the right has conductivity $\frac{1}{2} \gamma(e)-\alpha(e')$.
\end{enumerate}

It is easy to formalize the notion of a directed resistor network. A directed graph is a set $(V,D)$ where $V$ (vertices) is a finite set, and $D$ (directed edges) is a set of ordered pairs of elements in $V$.

\begin{definition}
A \emph{directed resistor network} is a directed graph $(V,D)$ together with a distinguished set $V_B \subset V$ (boundary vertices) and a function $\mu: D \to \mR$.
\end{definition}

We now describe the construction of piecewise constant conductivities which correspond to a given triangular resistor network $(V,E,V_B,\gamma)$ with domain $\Omega \subset \mR^2$. Let $D_I$ be some choice of directed edges corresponding to $E_I$ as above, and let $\alpha$ be a function $D_I \to \mR$. We wish to find a symmetric matrix function $\sigma = \sigma_{\alpha}$ satisfying the following conditions:
\begin{enumerate}
\item[(a)] 
$\sigma$ is a constant matrix in each open triangle,
\item[(b)]
if $x_j$ is any interior vertex and $x_k \in \mathcal{N}(x_j)$, if $e'$ is the directed edge corresponding to $\{x_j,x_k\}$, and if $T$ and $T'$ are the two triangles which have both $x_j$ and $x_k$ as vertices such that $\partial T$ has the same orientation as $e'$ (so then $\partial T'$ has the opposite orientation as $e'$), then 
\begin{align*}
\int_T \sigma \nabla v_j \cdot \nabla v_k \,dx &= -\frac{1}{2}\gamma_{jk} - \alpha(e'), \\
\int_{T'} \sigma \nabla v_j \cdot \nabla v_k \,dx &= -\frac{1}{2}\gamma_{jk} + \alpha(e'),
\end{align*}
\item[(c)]
if $x_j$ and $x_k$ are neighboring boundary vertices, and if $T$ is the unique triangle which has both $x_j$ and $x_k$ as vertices, then 
\begin{equation*}
\int_T \sigma \nabla v_j \cdot \nabla v_k \,dx = -\gamma_{jk}.
\end{equation*}
\end{enumerate}

The next result shows that such a conductivity $\sigma = \sigma_{\alpha}$ always exists and it is uniquely determined by the above conditions.

\begin{thm}
\label{network fem equivalence2}
Let $(V,E,V_B,\gamma)$ be a triangular resistor network with domain $\Omega \subset \mR^2$ such that $0$ is not a Dirichlet eigenvalue, and let $D_I$ be as above. Then for any function $\alpha: D_I \to \mR$ there is unique symmetric matrix function $\sigma_{\alpha} \in L^{\infty}(\Omega ; \mR^{2 \times 2} )$ satisfying (a)--(c). Any such $\sigma_{\alpha}$ satisfies $L(\sigma_{\alpha}) = \tilde{L}(\gamma)$ and $\Lambda_{\sigma_{\alpha}}^{\text{FEM}} = \Lambda_{\gamma}$. Conversely, if $\sigma$ is a symmetric matrix function which is constant in each open triangle and satisfies $L(\sigma) = \tilde{L}(\gamma)$ and $\Lambda_{\sigma}^{\text{FEM}} = \Lambda_{\gamma}$, then $\sigma = \sigma_{\alpha}$ for some function $\alpha: D_I \to \mR$.
\end{thm}
\begin{proof}
Let $\alpha$ be a function $D_I \to \mR$. We define $\alpha$ also for boundary edges as follows: if $e = \{x_j,x_k\} \in E$ where both $x_j$ and $x_k$ are boundary vertices, let $T$ be the unique triangle having both $x_j$ and $x_k$ as vertices, let $e'$ be the directed edge corresponding to $e$ such that $\partial T$ has the same orientation as $e'$, and define $\alpha(e') = \frac{1}{2} \gamma_{jk}$.

We first show the existence of $\sigma_{\alpha}$ satisfying (a)--(c). Choose any triangle $T$, and let $x_j, x_k, x_l$ be its vertices. If $e_{jk}'$ is the directed edge corresponding to $e_{jk} = \{x_j,x_k\}$, define 
\begin{equation} \label{mujk_definition}
\mu_{jk} = \left\{ \begin{array}{cl} \frac{1}{2} \gamma_{jk} + \alpha(e_{jk}') & \quad \text{if $\partial T$ has the same orientation as $e_{jk}'$}, \\ \frac{1}{2} \gamma_{jk} - \alpha(e_{jk}') & \quad \text{otherwise}. \end{array} \right.
\end{equation}
The numbers $\mu_{kl}$ and $\mu_{lj}$ are defined analogously. Consider the three equations 
\begin{align}
\int_T \sigma \nabla v_j \cdot \nabla v_k \,dx &= -\mu_{jk}, \notag \\
\int_T \sigma \nabla v_k \cdot \nabla v_l \,dx &= -\mu_{kl}, \label{sigma_int_triangle_equations} \\
\int_T \sigma \nabla v_l \cdot \nabla v_j \,dx &= -\mu_{lj}. \notag
\end{align}
Since $\sigma$ and each $\nabla v_m$ are constant in $T$, these are equivalent with 
\begin{align}
\sigma \nabla v_j \cdot \nabla v_k &= -\mu_{jk}/\abs{T}, \notag \\
\sigma \nabla v_k \cdot \nabla v_l &= -\mu_{kl}/\abs{T}, \label{sigma_triangle_equations} \\
\sigma \nabla v_l \cdot \nabla v_j &= -\mu_{lj}/\abs{T}. \notag
\end{align}

Writing $\sigma = \left( \begin{smallmatrix} a & b \\ b & c \end{smallmatrix} \right)$ in $T$ for some real constants $a,b,c$, \eqref{sigma_triangle_equations} becomes a linear system of three equations in three unknowns (the constants $a,b,c$). The corresponding homogeneous system is 
\begin{align*}
\sigma \nabla v_j \cdot \nabla v_k &= 0, \\
\sigma \nabla v_k \cdot \nabla v_l &= 0, \\
\sigma \nabla v_l \cdot \nabla v_j &= 0.
\end{align*}
This only has the zero solution $\sigma = 0$, since for instance $\nabla v_k$ and $\nabla v_l$ are linearly independent so for any $w \in \mR^2$ one obtains 
\begin{equation*}
\sigma \nabla v_j \cdot w = \sigma \nabla v_j \cdot (w_1 \nabla v_k + w_2 \nabla v_l) = 0.
\end{equation*}
Consequently $\sigma \nabla v_j = 0$, and analogously $\sigma \nabla v_k = \sigma \nabla v_l = 0$ which implies $\sigma = 0$. This shows that the original system \eqref{sigma_int_triangle_equations} has a unique solution, and so we obtain a constant matrix $\sigma$ in $T$ satisfying \eqref{sigma_int_triangle_equations}.

We denote by $\sigma = \sigma_{\alpha}$ the matrix in $\Omega$ which is defined by the above procedure in each triangle. Clearly (a) is satisfied, and also (b) and (c) are valid by the definition of $\sigma_{\alpha}$. Also, if $\tilde{\sigma}$ is another matrix satisfying (a)--(c) then $\sigma-\tilde{\sigma}$ satisfies (a)--(c) with all right hand sides replaced by zero, and localizing to each triangle and using the uniqueness of solutions to $3 \times 3$ systems as above implies that $\sigma = \tilde{\sigma}$.

Note that one has  $L(\sigma) = \tilde{L}(\gamma)$ and $\Lambda_{\sigma}^{\text{FEM}} = \Lambda_{\gamma}$ if and only if 
\begin{equation} \label{sigma_gamma_cond1}
\int_{\Omega} \sigma \nabla v_j \cdot \nabla v_k \,dx = -\gamma_{jk}, \qquad j \geq B+1, \ \ k \neq j.
\end{equation}
This statement follows by writing out the definitions of the matrices and Dirichlet-to-Neumann maps and using that $\sum_{m=1}^N \int_{\Omega} \sigma \nabla v_j \cdot \nabla v_m \,dx = 0$. To prove \eqref{sigma_gamma_cond1}, if at least one of $x_j$, $x_k$ is an interior vertex, there are two triangles $T$ and $T'$ having both $x_j$ and $x_k$ as vertices. We choose $T$ to be the triangle such that $\partial T$ has the same orientation as the directed edge $e'$ corresponding to $\{x_j,x_k\}$. One has 
\begin{align*}
\int_{\Omega} \sigma \nabla v_j \cdot \nabla v_k \,dx &= \int_T \sigma \nabla v_j \cdot \nabla v_k \,dx + \int_{T'} \sigma \nabla v_j \cdot \nabla v_k \,dx \\
 &= \left[ -\frac{1}{2} \gamma_{jk} -\alpha(e') \right] + \left[ -\frac{1}{2} \gamma_{jk} + \alpha(e') \right] = -\gamma_{jk}
\end{align*}
by the construction of $\sigma_{\alpha}$. On the other hand, if both $x_j$ and $x_k$ are boundary vertices, there is a unique triangle $T$ having both these points as vertices, and again by construction 
\begin{equation*}
\int_{\Omega} \sigma \nabla v_j \cdot \nabla v_k \,dx = \int_T \sigma \nabla v_j \cdot \nabla v_k \,dx = -\gamma_{jk}.
\end{equation*}
Thus $L(\sigma) = \tilde{L}(\gamma)$ and $\Lambda_{\sigma}^{\text{FEM}} = \Lambda_{\gamma}$.

Finally we prove that for any symmetric matrix $\sigma$ which is constant in each open triangle and satisfies $L(\sigma) = \tilde{L}(\gamma)$ and $\Lambda_{\sigma}^{\text{FEM}} = \Lambda_{\gamma}$, one has $\sigma = \sigma_{\alpha}$ for some $\alpha: D_I \to \mR$. This is easy at this point. The assumption implies that \eqref{sigma_gamma_cond1} holds. Let $e' \in D_I$, and let $e = \{x_j,x_k\}$ be the corresponding undirected edge where $x_j$ is an interior vertex. There are two triangles $T$ and $T'$ having both $x_j$ and $x_k$ as vertices, chosen so that $\partial T$ has the same orientation as $e'$. Note that by \eqref{sigma_gamma_cond1} 
\begin{equation*}
\int_T \sigma \nabla v_j \cdot \nabla v_k \,dx + \int_{T'} \sigma \nabla v_j \cdot \nabla v_k \,dx = -\gamma_{jk}.
\end{equation*}
We define 
\begin{equation*}
\alpha(e') = - \frac{1}{2} \gamma_{jk} -\int_T \sigma \nabla v_j \cdot \nabla v_k \,dx.
\end{equation*}
Then 
\begin{equation*}
\int_{T'} \sigma \nabla v_j \cdot \nabla v_k \,dx = -\gamma_{jk} - \int_T \sigma \nabla v_j \cdot \nabla v_k \,dx = - \frac{1}{2} \gamma_{jk} + \alpha(e').
\end{equation*}

It follows that with this definition of $\alpha$, $\sigma$ satisfies conditions (a) and (b) above. If $x_j$ and $x_k$ are boundary vertices, there is a unique triangle $T$ having both these points as vertices and therefore by \eqref{sigma_gamma_cond1} 
\begin{equation*}
\int_T \sigma \nabla v_j \cdot \nabla v_k \,dx = -\gamma_{jk}.
\end{equation*}
Thus $\sigma$ also satisfies (c). Since $\sigma$ is uniquely determined by (a)--(c), it follows that $\sigma = \sigma_{\alpha}$.
\end{proof}

\begin{proof}[Proof of Theorem \ref{network fem equivalence}]
Theorem \ref{network fem equivalence} follows from the previous theorem after fixing some choice of directed interior edges $D_I$ corresponding to the set $E_I$ of non-directed interior edges, and after identifying the function $\alpha: E_I \to \mR$ with $\alpha: D_I \to \mR$.
\end{proof}

\begin{remark}
An isomorphism of two directed graphs $(V,D)$ and $(V',D')$ is a bijective map $\phi: V \to V'$ for which $(x,y) \in D$ if and only if $(\phi(x), \phi(y)) \in D'$. Note that an isomorphism $\Phi$ of two triangulations $\mathscr{T}$ and $\mathscr{T}'$ induces an isomorphism of the associated directed graphs, and conversely two triangulations $\mathscr{T}$ and $\mathscr{T}'$ whose associated graphs are isomorphic via some $\phi$ are isomorphic via a map $\Phi$ that is uniquely determined by $\varphi$. Indeed, the map $\Phi$ is uniquely determined by its values at the vertex points.
\end{remark}

The next result gives a condition that characterizes when the matrix $\sigma_{\alpha}$ above is positive semidefinite.

\begin{lemma}
\label{positivity}
The conductivity matrix $\sigma$ constructed in Theorem \ref{network fem equivalence} is positive semidefinite if and only if on each triangle $T$ with vertices $x_j$, $x_k$, and $x_l$, the conductivities $\mu_{j,k}$, $\mu_{l,j}$ and $\mu_{k,l}$  that appear in the proof of the theorem  and are associated with the directed edges of the triangle are chosen to satisfy 
\begin{gather*}
\mu_{jk} + \mu_{kl} + \mu_{lj} \geq 0, \\
\mu_{jk}\mu_{kl} + \mu_{jk} \mu_{lj} + \mu_{kl} \mu_{lj} \geq 0.
\end{gather*}
\end{lemma}
\begin{proof}
Let $T$ be a triangle with vertices which we assume without loss of generality are $x_1$, $x_2$, and $x_3$. Assume first the special case that $T$ is a a equilateral triangle, and choose coordinates so that 
$$
x_1 = \left( \begin{array}{c} -l/\sqrt{3} \\ 0 \end{array} \right), \quad x_2 = \left( \begin{array}{c} l/\sqrt{3} \\ 0 \end{array} \right), \quad x_3 = \left( \begin{array}{c} 0 \\ l \end{array} \right)
$$
where $l > 0$. Then $\abs{T} = l^2/\sqrt{3}$, and in $T$ we have 
$$
\nabla v_1 = \frac{1}{l} \left( \begin{array}{c} -\sqrt{3}/2 \\ -1/2 \end{array} \right), \quad \nabla v_2 = \frac{1}{l} \left( \begin{array}{c} \sqrt{3}/2 \\ -1/2 \end{array} \right), \quad \nabla v_3 = \frac{1}{l} \left( \begin{array}{c} 0 \\ 1 \end{array} \right).
$$
Solving the equation (\ref{sigma_triangle_equations}) for $\sigma = \left( \begin{smallmatrix} a & b \\ b & c \end{smallmatrix} \right)$ we get that
\begin{gather*}
a = \frac{1}{\sqrt{3}}(\mu_{3,1} + \mu_{2,3}) + \frac{4}{\sqrt{3}} \mu_{1,2}, \\
b = \mu_{3,1} - \mu_{2,3}, \\
c = \sqrt{3}(\mu_{3,1} + \mu_{2,3}).
\end{gather*}
With these entries, the trace and determinant of $\sigma$ satisfy 
\begin{gather*}
\text{tr}(\sigma) = \frac{4}{\sqrt{3}} (\mu_{1,2} + \mu_{2,3} + \mu_{3,1}), \\
\det(\sigma) = 4(\mu_{1,2} \mu_{2,3} + \mu_{1,2} \mu_{3,1} + \mu_{2,3} \mu_{3,1}).
\end{gather*}
Since $\sigma$ is real and symmetric, we have that $\sigma$ is positive semidefinite if and only if the trace and determinant are $\geq 0$. This is equivalent with the conditions in the statement.

Now for a general triangle $T$, find an affine diffeomorphism $F$ such that $T' := F(T)$ is an equilateral triangle. Then make a change of variables using $F$ to the equation (\ref{sigma_int_triangle_equations}) to conclude that $F_* \sigma$ is positive definite if and only if conditions in this lemma are met. However, $F_* \sigma$ is positive definite if and only if $\sigma$ is positive definite.
\end{proof}

To end this section, we show how the double resistor network construction can be used to prove Theorems \ref{thm_main_nonuniqueness1} and \ref{thm_main_nonuniqueness2}.

\begin{proof}[Proof of Theorem \ref{thm_main_nonuniqueness1}]
Assume that we are given a triangulation $\mathscr{T}$ of a domain $\Omega \subset \mR^2$ that has associated graph $(V,E)$ with interior edges $E_I$ and boundary edges $E_B$, and that we are also given a conductivity matrix $\sigma \in PC_+(\mathscr{T})$ and a function $\beta: E_I \to \mR$.

Recall from Section \ref{sec_intro} that we are looking for a conductivity $\tilde{\sigma} = S_{\beta}(\sigma)$ that satisfies the following three conditions:
\begin{enumerate}
\item[(i)] 
$\tilde{\sigma}$ is a constant matrix in the interior of each triangle in $\mathscr{T}$,
\item[(ii)]
if $e = \{x_j,x_k\} \in E_I$ and $j < k$, and if $T$ and $T'$ are the two triangles of $\mathscr{T}$ having both $x_j$ and $x_k$ as vertices such that $\partial T$ has the same orientation as the directed edge $d = (x_j,x_k)$ (so then $\partial T'$ has the opposite orientation from $d$), then 
\begin{align*}
\int_T \tilde{\sigma} \nabla v_j \cdot \nabla v_k \,dx &= \int_T \sigma \nabla v_j \cdot \nabla v_k \,dx + \beta(e), \\
\int_{T'} \tilde{\sigma} \nabla v_j \cdot \nabla v_k \,dx &= \int_{T'} \sigma \nabla v_j \cdot \nabla v_k \,dx - \beta(e),
\end{align*}

\item[(iii)]
if $e = \{x_j, x_k\} \in E_B$, and if $T$ is the unique triangle which has both $x_j$ and $x_k$ as vertices, then 
\begin{equation*}
\int_T \tilde{\sigma} \nabla v_j \cdot \nabla v_k \,dx = \int_T \sigma \nabla v_j \cdot \nabla v_k \,dx.
\end{equation*}
\end{enumerate}
We define a set of directed edges corresponding to $E_I$ by 
$$
D_I = \{ e' = (x_j, x_k) \,;\, e = \{ x_j, x_k \} \in E_I, \ \ j < k \}.
$$
We will then choose $\tilde{\sigma} = \sigma_{\alpha}$ as in Theorem \ref{network fem equivalence2} for some function $\alpha: D_I \to \mR$ that will depend on $\beta$.

Let us begin by constructing the resistor network $(V,E,V_B,\gamma)$ that is equivalent with the FEM model as in Theorem \ref{thm_main_fem_to_resistor_general} and satisfies $\tilde{L}(\gamma) = L(\sigma)$ and $\Lambda_{\gamma} = \Lambda_{\sigma}^{\mathrm{FEM}}$. This was obtained by taking 
$$
\gamma_{jk} = -\int_{\Omega} \sigma \nabla v_j \cdot \nabla v_k \,dx, \qquad 1 \leq j \leq N, \ \ k \neq j.
$$
(See the proof of Theorem \ref{thm_main_fem_to_resistor_general}.)

Let now $e = \{x_j,x_k\} \in E_I$ where $j < k$, and let $T$ and $T'$ be the two triangles of $\mathscr{T}$ having both $x_j$ and $x_k$ as vertices such that $\partial T$ has the same orientation as the directed edge $e' = (x_j,x_k)$ (so $\partial T'$ has the opposite orientation from $e'$). Looking at the supports of $\nabla v_j$ and $\nabla v_k$, we have 
$$
-\gamma_{jk} = \int_{T} \sigma \nabla v_j \cdot \nabla v_k \,dx + \int_{T'} \sigma \nabla v_j \cdot \nabla v_k \,dx.
$$
We make the choice 
$$
\alpha(e') = -\int_{T} \sigma \nabla v_j \cdot \nabla v_k \,dx - \beta(e) - \frac{1}{2} \gamma_{jk}.
$$
By using the expression for $-\gamma_{jk}$ above, we have 
\begin{align*}
-\frac{1}{2} \gamma_{jk} - \alpha(e') &= \int_{T} \sigma \nabla v_j \cdot \nabla v_k \,dx + \beta(e), \\
-\frac{1}{2} \gamma_{jk} + \alpha(e') &= \int_{T'} \sigma \nabla v_j \cdot \nabla v_k \,dx - \beta(e).
\end{align*}

The above construction defines a function $\alpha: D_I \to \mR$, and Theorem \ref{network fem equivalence2} guarantees that there is a conductivity $\sigma_{\alpha} \in PC(\mathscr{T})$ that satisfies conditions (a)--(c) and also $L(\sigma_{\alpha}) = \tilde{L}(\gamma) = L(\sigma)$ and $\Lambda_{\sigma_{\alpha}}^{\text{FEM}} = \Lambda_{\gamma} = \Lambda_{\sigma}^{\text{FEM}}$. We define $\tilde{\sigma} = \sigma_{\alpha}$. The conditions (a)--(c) imply (i)--(iii), and we also have $L(\tilde{\sigma}) = L(\sigma)$ and $\Lambda_{\tilde{\sigma}}^{\text{FEM}} = \Lambda_{\sigma}^{\text{FEM}}$.
\end{proof}

\begin{proof}[Proof of Theorem \ref{thm_main_nonuniqueness2}]
By Theorem \ref{thm_main_nonuniqueness1}, for any $t \in \mR$ we have 
$$
\Lambda_{\sigma}^{\text{FEM}(\mathscr{T})} = \Lambda_{S_{t\beta}(\sigma)}^{\text{FEM}(\mathscr{T})}.
$$
If $\Phi_t$ is as in the statement of the theorem and if $t$ is sufficiently small, then there is a triangulation $\mathscr{T}_t$ of $\Omega$ that is isomorphic to $\mathscr{T}$ via $\Phi_t$. Also $(\Phi_t)_* S_{t\beta}(\sigma)$ is in $PC_+(\mathscr{T})$ for $t$ small. Theorem \ref{main_thm_diffeo_invariance} then implies that 
$$
\Lambda_{S_{t\beta}(\sigma)}^{\text{FEM}(\mathscr{T})} = \Lambda_{(\Phi_t)_* S_{t\beta}(\sigma)}^{\text{FEM}(\mathscr{T_t})}.
$$
This proves the result.
\end{proof}

\begin{section}{Cloaking for discrete and continuous models} \label{sec_cloaking}

By invisibility cloaking one means the possibility, 
both theoretical and practical, of shielding a
region or object from
detection via electromagnetic fields.  
The invisibility cloaking examples,
see \cite{GKLLU,GKLU1,GLU3,MN,Milton2,Le,PSS1}, have been the starting point for transformation optics,  a very topical 
subject in recent studies in mathematics, physics, and 
material science.

To consider invisibility cloaking results for discrete models, let us recall the corresponding results for  the 
anisotropic conductivity equation (\ref{johty}) in $\Omega\subset \mR^2$,
where the
 conductivity
$\sigma=[\sigma^{jk}(x)]_{j,k=1}^2$ is  a measurable function
which values are symmetric, positive definite matrixes.  
We say that a conductivity $\sigma$  is regular if there are 
$c_1,c_2>0$ such that the matrix $\sigma (x)$ satisfies
\ba
c_1\II \leq \sigma(x)\leq c_2\II,\quad \hbox{for a.e.  }x\in \Omega.
\ea
If conductivity is not regular, it is said to be degenerate.

Let us first  recall the
counterexamples for the uniqueness of the inverse problem that have a close
connection to the invisibility cloaking.
In 2003,
before the appearance of practical possibilities for cloaking,
 it was shown in  \cite{GLU2,GLU3} that passive
objects  can be coated with a layer of material with a degenerate
conductivity which makes the object undetectable by 
the electrostatic boundary
measurements.
These constructions were 
based on the blow-up maps
and gave counterexamples for the uniqueness of  inverse conductivity
problem in the three and
higher dimensional cases.
In two dimensional case, the mathematical theory of the cloaking examples for conductivity
equation have been studied in \cite{KOVW,KSVW,LZ,Ng}.

The interest in cloaking was
raised in particular in 2006 when it was realized
that practical
cloaking constructions are possible using so-called metamaterials
which allow fairly arbitrary specification of electromagnetic material
parameters. The construction of Leonhardt
\cite{Le} was based on conformal mapping on a non-trivial Riemannian surface. 
At the same time, Pendry et al
\cite{PSS1}  proposed a cloaking construction
for Maxwell's equations using a blow-up map 
 and the idea was demonstrated in laboratory experiments
\cite{SMJCPSS}.   Cloaking constructions for the conductivity equation based on the mathematical
models in \cite{KSVW,GLU2} have
also been recently verified experimentally in \cite{dc-cloak}.

Let $\Sigma=\Sigma(\Omega)$ be the class of measurable matrix valued functions 
$\sigma: \Omega \to M$, where $M$ is the set of 
symmetric positive semi-definite matrices.
Instead of defining the   Dirichlet-to-Neumann operator which may not be well defined
for these conductivities, we consider the corresponding quadratic forms.

\begin{definition}\label{def: 1A} {\rm Let $h\in H^{1/2}(\partial \Omega)$.
The Dirichlet-to-Neumann quadratic form corresponding to the conductivity
$\sigma\in \Sigma(\Omega)$ is
given by
\beq\label{def: Dirichlet-to-Neumann map}
Q_\sigma[h]=\inf A_\sigma[u],\quad\hbox{where }
A_\sigma[u]=\int_\Omega \sigma(z)\nabla u(z)\, \,\cdotp\nabla u(z)\,\,dz,
\eeq
and the infimum is taken over real valued
$u\in L^1(\Omega)$ such that $ \nabla u\in L^1(\Omega)^3$
and $u|_{\partial \Omega}=h$. In the case where $Q_\sigma[h] < \infty$ and $A_\sigma[u]$ reaches
its minimum at some $u$, we say that $u$ is a $W^{1,1}(\Omega)$
solution of the conductivity problem.}
\end{definition}

In the case when $\sigma$ is smooth, bounded from below and above by positive
constants, we can define a Dirichlet-to-Neumann
map (\ref{eq: DN map}).
For such regular conductivities
$Q_\sigma[h]$ is the quadratic form corresponding the Dirichlet-to-Neumann
map, that is, 
\beq\label{eq: Q forms}
Q_\sigma[h] =  \int_{\partial \Omega} h \, \Lambda_{\sigma} h\,  ds ,
\eeq
where $ds$ is the length measure on $\p \Omega$.
Physically, $Q_\sigma[h]$ corresponds to the power needed to keep voltage
$h$ at the boundary.  

Let us next consider various counterexamples for the solvability of inverse conductivity problem
with degenerate conductivities.

\medskip

\noindent
\begin{subsection}{Invisibility cloaking for the continuous model}
In this subsection we recall the invisibility  cloaking example  based on the coordinate transformation
by a blow-up map \cite{GLU3}.
Let us consider the general background conductivity $\sigma$.
We aim to coat an arbitrary  body with a layer of exotic material so that the coated  body
 appears in measurements the same as the background conductivity $\sigma$. Usually one
 is interested in the case when  the background conductivity $\sigma$  is equal to
the constant $\gamma=1$. However, we consider here a more general
case and assume that $\sigma$ is a $L^\infty$-smooth conductivity in a domain $\Omega\subset \mR^2$
that contains the closed disc $\overline \ID(2)$. {Here, $\ID(\rho)$ is an open 2-dimensional disc of radius $\rho$ and center zero
and $\overline \ID(\rho)$ is its closure. 
Also, assume that 
$\sigma(z)\geq c_0I,$ $c_0>0$.
Consider 
the homeomorphism
\beq\label{general blow up}
F:\overline \Omega \setminus\{0\}\to
\overline \Omega
\setminus \ID(1),
\eeq
given by
\beq\label{blow up}
 & &F(z)=(\frac {|z|}2+1)\frac z{|z|},\quad z\in \ID(2),\\
 & &F(z)=z,\quad z\in \overline \Omega\setminus  \ID(2).\nonumber
\eeq

Using the map (\ref{blow up}), we  
define a singular conductivity
\beq\label{eq: sing cond}
\tilde \sigma (z)=\left\{\begin{array}{cl}
(F_*\sigma)(z)
   & \hbox{for }z\in \overline \Omega\setminus \ID(1),\\
\eta(z) & \hbox{for }z\in  \ID(1),\end{array}\right.
\eeq
where $\eta(z)=[\eta^{jk}(x)]$ is any symmetric measurable matrix satisfying
$c_1I\leq \eta(z)\leq c_2I$ with $c_1,c_2>0$. 
The conductivity $F_*\sigma$ is called the cloaking conductivity obtained from
the transformation map $F$ and background conductivity $\sigma$ and $\eta(z)$
is the conductivity of the cloaked (i.e.\ hidden) object.

In dimensions $n\geq 3$ it shown in 2003 in \cite{GLU2,GLU3}
{that the Dirichlet-to-Neumann map corresponding to $H^1(\Omega)$ solutions
for the conductivity (\ref{eq: sing cond}) coincide with the
Dirichlet-to-Neumann map for  $\sigma$ when $\sigma=1$.}
In  2008, the analogous result for  Dirichlet-to-Neumann maps was proven in   the two-dimensional case in
\cite{KSVW}. Using the above formalism of Dirichlet-to-Neumann forms it was shown
in \cite{ALP2} that the
quadratic form $Q_{\tilde \sigma}$ coincides with $Q_{\sigma}$. This means that when $\sigma$ is a regular
 conductivity in the domain $\Omega$ containing the disc $\overline \ID(2)$, we can modify
 the conductivity in  $\overline \ID(2)$ to a conductivity $\tilde \sigma$ so that 
 $\ID(1)$  contains the object with conductivity
 $\eta$ and  the conductivity in $\overline \ID(2)\setminus \overline \ID(1)$ is such
 that for the  obtained conductivity $\tilde \sigma$ in $\Omega$  all measurements
 on $\p \Omega$ coincide with those corresponding to $\tilde \sigma$, that is,
 the object $\eta$ is hidden in the conducting medium and one can not detect from
 the boundary even the fact that something is hidden. 
 Due to the presence of degenerate coefficients in the partial differential equation, numerical simulation
of cloaking is a challenging (and interesting) problem, see e.g. \cite {HL,FEMc1,FEMc2}
and references therein. 
 Next we will consider cloaking starting directly from the FEM model
 for the background space and construct cloaking models using only
 the FEM models and the corresponding resistor networks. We study
in this paper only the static (i.e.\ zero frequency) cloaking and leave
open the problem how the corresponding constructions could be done
for the Helmholtz equation or time-harmonic Maxwell's equations.  }
\end{subsection}

\noindent
\begin{subsection}{ Invisibility cloaking for the discrete model}

In this subsection we construct resistor networks and their equivalent FEM models with cloaking properties. Consider first the resistor network with graph given by
\begin{center}


\begin{tikzpicture}[scale = 0.7]
    \tikzstyle{every node}=[draw,circle,fill=white,minimum size=4pt,
                            inner sep=0pt]

    \draw (0,0) node (1) [label=above:$ x_{1}$] {}
        -- ++(0:4.0cm) node (3) [label=above:$ x_{3}$] {}
        -- (1) 
        -- ++(120:-4.0cm) node (2) [label=above:$ x_{2}$] {}
        -- ++(3);
\draw(1)
--++ (60:4cm) node(5)[label = above: $x_5$]{};
\draw --++ (5) --++(3);
\draw --++(3) --++(1);
\draw(5)
 --++ (4,0cm) node(4) [label=right:$x_4$] {} ;
 \draw(5)
 --++(180: 4cm) node(6) [label = above: $x_6$]{};
\draw --++ (4) --++(3);
\draw --++ (1) --++(6);
\end{tikzpicture}
\end{center}
{We call this graph consisting of four equilateral triangles the graph of type $B$ (B refers to background) and
say that $(x_1x_3x_5)$ is the basic triangle of the graph.}
We consider the graph B as a resistor network by assigning 
 resistivity $1$ on all edges of the graph. 
 
Also, we consider the second  graph, shown below, that is an extension
of the above graph of type B.
\begin{center}
\begin{tikzpicture}[scale = 0.7]
    \tikzstyle{every node}=[draw,circle,fill=white,minimum size=4pt,
                            inner sep=0pt]

    \draw (0,0) node (1) [label=above:$ x_{1}$] {}
        -- ++(0:2.0cm) node (8) [label=below:$ x_{8}$] {}
        -- (1) 
        -- ++(300:4.0cm) node (2) [label=above:$ x_{2}$] {};

\draw (8)
        -- ++(0:2.0cm) node (3) [label=above:$ x_{3}$] {};
\draw--++(3) --++(2);
\draw(3)
--++ (120: 2.0cm) node(9) [label= right: $x_9$]{};

\draw(1)
--++ (60:2cm) node(7)[label = left: $x_7$]{};
\draw(7)
--++ (60:2cm) node(5)[label = above: $x_5$]{};
\begin{scope}[ultra thick]
\draw --++(9) --++(5);
\draw (9) --++(300: 1.95cm);
\end{scope}
\draw(5)
 --++ (4,0cm) node(4) [label=right:$x_4$] {} ;
 \draw(5)
 --++(180: 4cm) node(6) [label = above: $x_6$]{};
\draw --++ (4) --++(3);
\draw --++ (1) --++(6);
\draw (2,0.5) node(10) [label = right :$x_{10}$]{};
\draw (1.45,1.47) node(11)[label =above: $x_{11}$ ]{};
\draw (2.55,1.47) node(12)[label =above: $x_{12}$ ]{};

\begin{scope}[dashed]
\draw --++ (6) --++(7);
\draw --++ (2) --++(8);
\draw --++ (4) --++(9);
\draw (8) --++(90: 0.5cm);
\draw (7) --++ (330:0.5cm);
\draw (9) --++ (210: 0.5cm);
\draw --++(1) --++(10);
\draw --++(1) --++(11);
\draw --++(5) --++(11);
\draw --++(3) --++(10);
\draw --++(3) --++(12);
\end{scope}
\draw --++(5) --++(12);

\begin{scope}[ultra thick]
\draw --++(5) --++(7) ;
\draw --++ (7) --++(1);
\draw --++(1) --++ (8);
\draw --++(8) --++(3);
\end{scope}
\begin{scope}[ultra thick, dotted]
\draw --++(10) --++(11);
\draw --++(12) --++(10);
\draw (12) --++(11);
\end{scope}
\end{tikzpicture}
\end{center}

This graph is formed by choosing points 
$x_{10}$, $x_{11}$, and $x_{12}$ to be on the line segments connecting
the points $x_{2}$, $x_{6}$, and $x_{4}$ to the centroid  $A$ of 
the triangle  $(x_{2},x_{6},x_{4})$ so that the points 
$x_{10}$, $x_{11}$, and $x_{12}$ have the equal distances to $A$.

Note that in this graph no triangle has an angle $> \pi/2$.
We consider  also this graph as a resistor network by assigning 
the resistivity by the following rules. On regular lines the resistivity is $1$, on dashed lines the resistivity is zero, on bold lines the resistivity is $2$, and on the inner triangle the dotted edges have arbitrary strictly positive resistivity.  Observe that the edge $x_5x_{12}$ has the resistivity $1$. 
We call this graph consisting of 16 triangles the graph of type $C$ (C refers to cloak) and
say that $(x_1x_3x_5)$ is the {\it basic triangle} corresponding to the {\it center triangle} $(x_{10}x_{11}x_{12})$.

In the following, we consider one fixed graph of type B and another
graph of type C that is an extension of the graph B, and refer to these graphs and the corresponding
resistor networks and 
as the graph B and C and the resistor network B and C.

By
using the construction given in the proofs of Theorems  \ref{network fem equivalence}
and \ref{network fem equivalence2}
with $\alpha_{j,k} = \gamma_{j,k}/2$
we can construct a piecewise constant positive semi-definite conductivity $\sigma^C
\in PC(\mathscr{T}^C)$ with
triangularization $\mathscr{T}^C$ given by the graph C. Note that
the conductivity  $\sigma^C$ in everywhere non-zero.
This gives a FEM model
 that corresponds to the resistor network C. Note that the constant
 conductivity $\sigma^B=1$ with
triangularization  $\mathscr{T}^B$ given by the graph B gives a FEM model
 that  corresponds to the resistor network B.

\begin{prop}
\label{cloak}
The resistor networks B and C with the 
set of boundary vertices $V_B=\{x_1,x_2,x_3,x_4,x_5,x_6\}$
have the same Dirichlet-Neumann map. Moreover, the corresponding
 the Dirichlet-Neumann maps  for the FEM models satisfy
 \beq\label{eq: cloaking FEM}
\Lambda_{\sigma^C}^{\text{FEM}( \mathscr{T}^C)} = \Lambda_{\sigma^B}
^{\text{FEM}( \mathscr{T}^B)}.
\eeq
\end{prop}
\begin{proof} 
Both graphs B and C, when edges with zero conductivity is removed, remain connected graphs. Therefore, the conductivity boundary value problem for both graphs are well-posed and the Dirichlet-Neumann map is well defined. Let $\Lambda_1$ and $\Lambda_2$ be the Dirichlet-to-Neumann map associated with the  graph B and C, respectively.\\
Explicit computation shows that for all $f : V_B \to \mR$, $(\Lambda_1 f)(x_j) = (\Lambda_2 f)(x_j)$ for $j = 1,2,3,4,6$. Now 
\begin{eqnarray}
\label{eq at x5}
(\Lambda_2 f)(x_5) = (\Lambda_1 f)(x_5) + f(x_5) - u(x_{12}).
\end{eqnarray}
However the conductivity equation states that 
\ba
& & f(x_5) - u(x_{12}) = \gamma_{10,12}(u(x_{12}) - u(x_{10}))+  \gamma_{11,12}(u(x_{12}) - u(x_{11})),
\\
& & 0 =  \gamma_{10,11}(u(x_{11}) - u(x_{10}))+  \gamma_{11,12}(u(x_{11}) - u(x_{12})),\\
& & 0 =  \gamma_{10,12}(u(x_{10}) - u(x_{12}))+  \gamma_{10,11}(u(x_{10}) - u(x_{11}))\ea
These three equations combine to give that $ f(x_5) - u(x_{12}) = 0$. With this information we can deduce from (\ref{eq at x5}) that $(\Lambda_2 f)(x_5) = (\Lambda_1 f)(x_5) $.

By the above,
the resistor networks B and C
have the same Dirichlet-Neumann map. Equation
(\ref{eq: cloaking FEM}) follows from this by Theorem  \ref{network fem equivalence}.
\end{proof}

Proposition \ref{cloak} means that in the FEM model corresponding to $\sigma^C$ the center triangle is cloaked. Below we use this conductivity to hide more general bodies.

\end{subsection}

\begin{figure}
\centerline{\includegraphics[width=9cm]{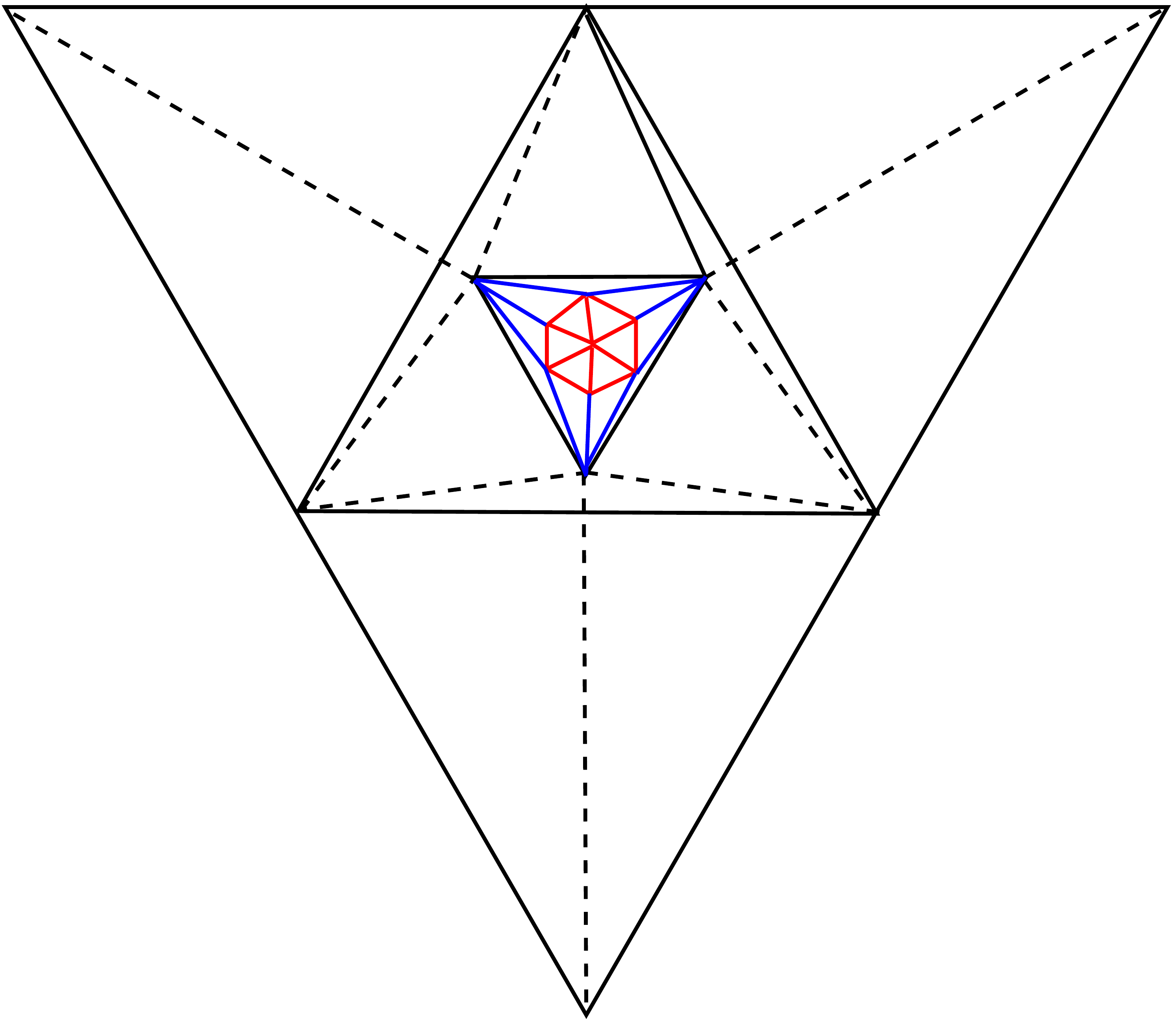}}
\caption{In the figure a general body $\Omega_*$ is coated with a cloaking 
structure. The red triangles are  a triangulation of $\Omega_*$. This triangulation
is extended with blue edges to form a triangulation of the triangle $T_0$
that is the center triangle in the resistor network corresponding to the graph C.
Transforming the resistor network to a FEM model, we obtain a conductivity where
 the piecewise constant conductivity $\sigma_*$ in the domain $\Omega_*$ 
 is surrounded with a positive-semidefinite conductivity that cloaks the presence of
 the conductivity in $\Omega_*$.  }
\end{figure}
\begin{subsection}{Cloaking of a more complicated body}

Let us now consider 
 the four triangles $T_1,T_2,T_3,T_4$ that form the graph  $B$ with $T_1$ being the basic  triangle
 and let triangle $\overline\Omega\subset \mR^2$ be the union of these triangles.
 
Similarly to the above let 
 $\mathscr{T}^B$ be the triangulation $\{T_1,T_2,T_3,T_4\}$ of $\overline \Omega$ and
$\sigma^B$ be the  conductivity that is constant 1 in all these triangles.

Let us also consider the graph C with the same boundary vertexes as the graph B
and let $T_0$ denote the center triangle of the graph C.
 Recall that above we assigned arbitrary positive resistivities  on the edges of the center triangle of the graph C and
  constructed a FEM model with conductivity $\sigma^C
  \in PC(\mathscr{T}^C)$.

Let $\Omega_*\subset \mR^2$ be a convex polygonal domain having some triangulation $\mathscr{T}_{*}$ and 
assume that $\overline \Omega_*\subset
T_0^{int}$ where $T_0$ is the center triangle of the graph $C$.  Let $\sigma_*\in PC_+(\mathscr{T}_*)$ be  an arbitrary piecewise constant conductivity in $\Omega_*$.

Let $\mathscr{T}_0$ denote the 
triangulation of $T_0$ that contains all triangles in $\mathscr{T}_*$
and all triangles in $T_0\setminus \Omega_*$ that are obtained by
connecting all boundary  vertexes of $\Omega_*$ to the vertexes of
$T_0$ with such line segments that intersect $\Omega_*$ only at one end point
of the line segment. In Figure 3 the red triangles correspond to the triangles
in  $\mathscr{T}_*$ and the blue line segments and the boundary edges of $\Omega_*$
define the other triangles in  $\mathscr{T}_0$.

The pair $(\Omega_*,\sigma_*)$ corresponds to a  body we want to hide from observations.
Let us define in $T_0$ the conductivity $\sigma_0\in PC_+(\mathscr{T}_0)$ so that
$\sigma_0$ coincides with $\sigma_*$ in $\Omega_*$ and is an arbitrary
positive definite piecewise constant conductivity in $T_0\setminus \Omega_*$.
Let $\tilde {\mathscr{T}}$ be a  triangulation of  $\overline \Omega$ that 
is the union of $\mathscr{T}^C\setminus\{T_0\}$ and $\mathscr{T}_0$.
Finally, let $\tilde{ \sigma}\in PC(\tilde{ \mathscr{T}})$ by the conductivity
that coincides with $\sigma_0$ in $T_0$ and with the cloaking conductivity $\sigma^C$ in $\Omega\setminus T_0$.
The FEM model $(\tilde{  \mathscr{T}},\tilde{ \sigma})$
corresponds to a resistor network obtained by starting from a resistor network of type
$C$, with some positive resistivities on the edges of the central triangle $T_0$, and  adding some additional vertexes and edges inside the triangle $T_0$.
In the resistor network corresponding the FEM model $(\tilde{  \mathscr{T}},\tilde{ \sigma})$ 
the solution of the equation (\ref{resistor_dp}) is constant in all vertexes that
are inside or on the boundary of the triangle $T_0$.  
Thus the exactly same arguments that were used to prove Prop.\ \ref{cloak}
give the following result:

\begin{thm} The conductivity $\tilde{ \sigma}\in  PC(\tilde{\mathscr{T}})$
is positive semidefinite and 
satisfies $\tilde{ \sigma}|_{\Omega_*}=\sigma^*$.
Moreover,
for the FEM models $(\tilde{  \mathscr{T}},\tilde{ \sigma})$ 
and $(\mathscr{T}^B,{\sigma^B})$ all boundary measurements
coincide in the sense that 
\label{FEM-cloaking}
 \beq\label{eq: cloaking FEM general}
\Lambda_{\tilde{ \sigma}}^{\text{FEM}(\tilde{  \mathscr{T}})} = \Lambda_{\sigma^B}
^{\text{FEM}( \mathscr{T}^B)}.
\eeq
\end{thm}

Summarizing the above construction, we can construct   a piecewise constant conductivity
$\tilde \sigma \in PC(\tilde {\mathscr{T}})$ for which the Dirichlet-to-Neumann maps
for the FEM models $(\tilde {\mathscr{T}},\tilde \sigma)$ and $( \mathscr{T}^B, \sigma^B)$ 
coincide and $(\tilde { \mathscr{T}},\tilde \sigma)$ contains the body  $(\Omega_*,\sigma_*)$.
This example can be considered as discrete invisibility cloaking of the body  $(\Omega_*,\sigma_*)$.
Note that in this example the "cloaking layer" around the body  $(\Omega_*,\sigma_*)$ is quite
thick, but analogously to the PDE cloaking examples, one can consider 
a piecewise affine diffeomorphism  $\Phi: \overline{\Omega} \to \overline{\Omega}_1$ to some suitable
domain $\Omega_1$, such that $\Phi$  is identity in $\Omega_*$. 
Using such diffeomorphisms one can obtain other cloaking examples 
where the domain $\Omega_1$ is a smaller neighborhood of $\Omega_*$. However, we 
do not consider any such particular examples in this paper.

Note that  by Lemma \ref{positivity} the conductivity matrices in $\tilde \sigma$ on some triangles will be only positive semi-definite  but not strictly positive definite. Analogues to the continuous PDE-based cloaking in \cite{GKLU6,KSVW} we will remedy this by constructing an approximate cloaking resistor network which will allow us to obtain a positive definite equivalent FEM model.
Indeed, by setting in the above constructions all the dashed lines in graphs of type $C$ to have conductivity $\epsilon >0$, we obtain a positive definite
 conductivity $\tilde \sigma(\epsilon) \in  PC_+(\tilde{\mathscr{T}})$. Since  $ \Lambda_{\tilde{ \sigma}(\epsilon)}^{\text{FEM}(\tilde{  \mathscr{T}})}  |_{\epsilon=0}= \Lambda_{\sigma^B}
^{\text{FEM}( \mathscr{T}^B)}
$, we can argue by continuity that $$\lim\limits_{\epsilon\to 0} 
 \Lambda_{\tilde{ \sigma}(\epsilon)}^{\text{FEM}(\tilde{  \mathscr{T}})} = \Lambda_{\sigma^B}
^{\text{FEM}( \mathscr{T}^B)}.$$ 
Hence, approximate cloaking for  FEM models is possible with positive definite conductivities.

\end{subsection}

\end{section}

\end{document}